\documentclass[aps,prx,twocolumn,groupedaddress]{revtex4-2}
\usepackage{amsmath,amsfonts}
\usepackage{multirow}
\usepackage{algorithmic}
\usepackage{algorithm}
\usepackage{array}
\usepackage{textcomp}
\usepackage{stfloats}
\usepackage{url}
\usepackage{verbatim}
\usepackage[dvipdfmx]{graphicx}
\usepackage{color}
\usepackage{bm,bbm}
\usepackage{braket}
\usepackage[subrefformat=parens]{subcaption}
\usepackage{amssymb}
\usepackage{amsthm}
\usepackage{enumerate}
\usepackage[russian,english]{babel} 

\theoremstyle{definition}
\newtheorem{prop}{Proposition}
\newtheorem{lem}{Lemma}

\newtheorem{asm}{Assumption}

\newcommand{\lb}{\left(}
\newcommand{\rb}{\right)}

\newcommand{\T}[1]{\tilde{#1}}
\newcommand{\Req}[1]{(\ref{eq:#1})}

\newcommand{\NReq}[1]{\eqref{eq:#1}}
\newcommand{\Reqs}[2]{\eqref{eq:#1} and \eqref{eq:#2}}

\newcommand{\Rfig}[1]{Fig.\ \ref{fig:#1}}

\newcommand{\Rfigs}[2]{Figs.\ \ref{fig:#1} and \ref{fig:#2}}

\newcommand{\Rtab}[1]{Table\ \ref{tab:#1}}
\newcommand{\Lfig}[1]{\label{fig:#1}}
\newcommand{\Leq}[1]{\label{eq:#1}}
\newcommand{\Ltab}[1]{\label{tab:#1}}
\newcommand{\Rsec}[1]{Sec.~\ref{sec:#1}}
\newcommand{\Lsec}[1]{\label{sec:#1}}
\newcommand{\be}{\begin{eqnarray}}
\newcommand{\ee}{\end{eqnarray}}
\newcommand{\no}{\nonumber}

\begin{document}

\title{Approximate Modeling for Supercritical Galton--Watson Branching Processes with Compound Poisson--Gamma Distribution}

\author{Kyoya Uemura, Tomoyuki Obuchi, and Toshiyuki Tanaka}
\affiliation{Graduate School of Informatics, Kyoto University, Kyoto 606-8501, Japan}


\begin{abstract}
We study asymptotic properties of supercritical Galton--Watson (GW) branching processes in the asymptotic where the mean of the offspring distribution approaches 1 from above.
We show that the population-size distribution of the GW branching processes at a sufficiently large generation in this asymptotic can be approximated by a compound Poisson--gamma distribution. 
Numerical experiments revealed that the compound Poisson--gamma models were in good agreement with the corresponding GW models for sufficiently large generations under a reasonable parameter regime.
Our results can be regarded as supporting the use of the compound Poisson--gamma model as a model for cascaded multiplication processes.
\end{abstract}

\maketitle

\section{Introduction}
\Lsec{Intro}

A Galton--Watson (GW) branching process (also called a Bienaym\'{e}--Galton--Watson branching process) is a stochastic model that describes population evolution over generations, which was originally developed to study propagation and extinction probabilities of family names~\cite{Bienayme1845, Watson1875}.
In the GW process, each individual in the $n$th generation independently produces a random number of offspring in the $(n+1)$th generation.
Let $Z_n$ denote the population size of $n$th generation. 
Then, $(Z_n)_{n=0,1,\ldots}$ evolves as follows:
\begin{align}
Z_{n+1} = \sum_{i=1}^{Z_n} X_{n}^{(i)}, \quad n=0,1,\ldots,
\Leq{GW process}
\end{align}
where $X_{n}^{(i)}$ denotes the number of offspring produced by the individual $i\in\{1,\ldots,Z_n\}$ in the $n$th generation, and where $\{X_{n}^{(i)}\}$ are independent and identically distributed (i.i.d.) according to a certain offspring distribution with mean $\lambda$.
Unless otherwise stated, we assume $Z_0 \equiv 1$.
When $\lambda > 1$ (referred to as the supercritical case), the GW process provides a reasonable model for cascaded multiplication processes, where the population increases with expected gain $\lambda$ in each generation, such as electron multiplication.
As a result, the supercritical GW process is now widely applied to model the population dynamics of these multiplication processes in fields such as biology and physics~\cite{Tantillo2024, Ventura2024, Bogdanov2021, Burden2016, Neves2006, Taneyhill1999, Lombard1961, Dietz1965, Prescott1966, Dietz1978}. 

One of the main interests in the supercritical GW process is regarding the distribution of $Z_{n}$ when $n$ is large.
For instance, in the context of electron multiplication, this distribution can be interpreted as the single electron response (SER) distribution (also known as the pulse height distribution) of an electron multiplier (EM) with a sufficient number of dynodes, providing insights into properties of response signals of detectors using EMs.
A primary approach for studying the distribution of $Z_{n}$ for large $n$ is to consider an alternative random variable defined by normalizing $Z_{n}$ with the total gain $\mathbb{E}[Z_n] = \lambda^n$ up to the $n$th generation~\cite{Harris1963, Athreya1972}.
Let
\begin{align}
W_{n} = \frac{Z_{n}}{\lambda^{n}}.
\Leq{W_n}
\end{align}
One then has
\begin{align}
  \mathbb{E}[W_{n+1}\mid W_1,\ldots,W_n]
  &=\lambda^{-n-1}\mathbb{E}[Z_{n+1}\mid Z_1,\ldots,Z_n]
  \nonumber\\
  &=\lambda^{-n-1}\mathbb{E}[Z_{n+1}\mid Z_n]
  \nonumber\\
  &=\lambda^{-n}Z_n=W_n,
  \label{eq:martingale_eq}
\end{align}
which shows that the sequence $(W_n)_{n=0,1,\ldots}$ is a nonnegative martingale, and consequently, the martingale convergence theorem~\cite{Doob1953} tells us that, as $n\to\infty$, $(W_n)_{n=0,1,\ldots}$ converges to a random variable $W$ almost surely.
In the supercritical case, the limiting distribution $\mathbb{P}(W)$ is nondegenerate under certain mild conditions~\cite{KestenStigum1966, Athreya1972}, providing an indirect evaluation of the distribution of $Z_{n}$ for large $n$.

However, the limiting distribution $\mathbb{P}(W)$ is generally not available in a closed form, even when the offspring distribution is simple, such as the Poisson offspring distribution.
This difficulty makes it challenging to apply the GW process model to downstream data analyses.
In practice, more tractable models based on empirical insights into the dataset (e.g., non-negativity, overdispersion, and excess zeros) are commonly adopted~\cite{Bellamy1994, Dossi2000, Ibrahimi2014}.
For example, Bellamy et al.~\cite{Bellamy1994} proposed describing the operation of a photomultiplier as consisting of two independent processes: photodetection and amplification.
They then formulated the response of a photomultiplier by using a compound Poisson model, where the number of photoelectrons on the photodetector is modeled by a Poisson distribution, and the amplification process is characterized by an SER distribution.
Subsequent studies~\cite{KalousisDeAndreBaussanDracos2020, WengZhangWuMaXuQianWangChen2024} discussed heuristic modeling of the SER distribution using a gamma distribution,
resulting in a compound Poisson-gamma (CPG) distribution as a model for the photomultiplier response. 
However, the theoretical basis of these modeling approaches has rarely been discussed.

In this work, we study the asymptotics of supercritical GW processes as they approach criticality, that is, as the mean $\lambda$ of the offspring distribution approaches 1 from above (i.e., $\lambda \downarrow 1$).
In the asymptotic regime $\lambda \downarrow 1$ with $Z_0 \equiv 1$ (Condition I), we derive a probability distribution in closed form whose cumulants match those of the limiting distribution $\mathbb{P}(W)$ up to the leading-order terms.
It is also shown that the derived distribution belongs to a class of CPG distributions. This is significant from a practical perspective, as the CPG distributions can be reparameterized as the Tweedie family, a subclass of exponential dispersion models (EDMs)~\cite{Jorgensen1987} which are known to possess several practically desirable properties, similar to those of exponential families~\cite{BarndorffNielsen2014,Efron2023}.
Moreover, we present a similar result for the case where $Z_0$ follows a certain distribution with mean $\lambda_0$, where $\lambda_0$ does not need to be close to unity (Condition II).
Note that Condition II reflects the standard assumption in response signal modeling of EMs~\cite{Lombard1961, Dietz1965}, in which a particle incident to the first-stage dynode produces secondary electrons with a gain different from that of the subsequent stages. 
Numerical experiments revealed that these CPG models were in good agreement with the corresponding GW models for large $n$ under a reasonable parameter regime.
These results provide mathematical and empirical support for the use of CPG distributions as a model for cascaded multiplication processes.


The remainder of this paper is organized as follows.
Section~\ref{sec:Theoretical Results} discusses our main theoretical results
which states that the limiting distribution $\mathbb{P}(W)$
in the asymptotic $\lambda\downarrow1$ can be
approximated by a CPG distribution
under Conditions I and II.
Section~\ref{sec:Numerical Verification} verifies 
the main results via comparing the population-size distribution
$\mathbb{P}(Z_n)$ with large $n$ and the CPG distribution.
Section~\ref{sec:Discussion} discusses tail behaviors of the
limiting distribution $\mathbb{P}(W)$. 
It also discusses the diffusion approximation and a large-$n$ analysis,
both of which further support the use of the CPG distribution as a model of
the population-size distribution of the GW processes when $n$ is large. 
Section~\ref{sec:Conclusion} concludes this paper. 

\section{Theoretical Results}
\Lsec{Theoretical Results}

\subsection{Condition I: $Z_0 \equiv 1$}
\Lsec{Condition I}

In this subsection, we discuss the limiting distribution for the supercritical GW process when $Z_0 \equiv 1$ (Condition I).
Let $f(s)$ and $\phi(t)$ be the probability generating function and the moment generating function of the offspring distribution, respectively: 
\begin{align}
f(s) &= \mathbb{E}[s^{X_{n}^{(i)}}] = \sum_{k=0}^{\infty} \mathbb{P}(X_{n}^{(i)}=k) s^k,
\Leq{f}\\
\phi(t) &= \mathbb{E}[e^{tX_{n}^{(i)}}] = f(e^t).
\Leq{phi}
\end{align}
Let $f_{n}(s)$ denote the probability generating function of $Z_{n}$.
From \NReq{GW process}, one has 
\begin{align}
f_{n+1}(s) &= \mathbb{E}_{Z_{n+1}} [s^{Z_{n+1}}] \no \\
&= \mathbb{E}_{Z_{n}} \left[ \left\{ \mathbb{E}_{X_n^{(i)}}[s^{X_{n}^{(i)}}] \right\}^{Z_n} \right] ~~(\because \text{$\{X_n^{(i)}\}$ is i.i.d.}) \no \\
&= f_n\left( f(s) \right), \quad n\geq1,
\Leq{f_n recursion}
\end{align}
where
\begin{align}
f_0(s) = s; \quad f_1(s) = f(s).
\end{align}
By recursively applying \NReq{f_n recursion}, one can in principle obtain the probability generating function $f_n(s)$ of $Z_{n}$ for any $n$.
In practice, however, for large $n$ its formula has a higher-order nested structure as
\begin{align}
f_n(s) = (\overbrace{f\circ f\circ\cdots\circ f}^{\text{$n$ times}})(s),
\Leq{PGF_nested}
\end{align}
and this mathematical complexity hinders the application of the GW model to general data analysis unless $f$ has a very specific structure~\cite{Harris1963, Sevastyanov1971, SagitovLindo2016}.

Here, to evaluate the probability distribution of $Z_{n}$ for large $n$, we follow the classical approach in the theory of the GW branching processes and begin our discussion by considering the normalized random variable $W_n$ defined by \NReq{W_n}, instead of $Z_{n}$. 
According to the Kesten-Stigum theorem~\cite{KestenStigum1966} (see also \cite[Theorem I.10.1]{Athreya1972}), when $\lambda>1$, the sequence $(W_n)_{n=0,1,\ldots}$ converges almost surely (a.s.) to a nondegenerate random variable $W$, 
if and only if $\mathbb{E}[X_n^{(i)} \log X_n^{(i)}]<\infty$ holds.
For the limiting random variable $W$, one has $\mathbb{E}[W]=1$, 
and the moment generating function $M_{W}(t)$ of $W$ satisfies the following functional equation
\begin{align}
M_{W}(\lambda t) = f\left( M_{W}(t) \right).
\Leq{functional equation1}
\end{align}
It can further be rewritten in terms of the cumulant generating function 
$K_{W}(t)=\log M_{W}(t)$ of $W$ as
\begin{equation}
  K_{W}(\lambda t) = \log f(e^{K_{W}(t)})
  =\psi(K_{W}(t)),
\Leq{functional equation_1_2}
\end{equation}
where $\psi(t)=\log\phi(t)=\log f(e^t)$ denotes the cumulant generating function of the offspring distribution.
These functional equations lay the basis for studying 
properties of the limiting random variable $W$.

One can notice that
the functional equation~\eqref{eq:functional equation_1_2} is invariant under
scaling of $t$, which implies that the cumulant generating function
of $aW$ for an arbitrary $a>0$ satisfies the same functional equation. 
With a modest amount of foresight, we consider the rescaled random variable $\bar{W}$ defined as
\begin{align}
\bar{W} = (\lambda-1) W,
\Leq{def bar_W}
\end{align}
for which $\mathbb{E}(\bar{W})=\lambda-1$,
and study properties of $\bar{W}$ via 
the following functional equation corresponding to \Req{functional equation_1_2}:
\begin{align}
\bar{K}(\lambda t) = \psi(\bar{K}(t)),
\Leq{functional equation2}
\end{align}
where $\bar{K}(t)$ denotes the cumulant generating function of $\bar{W}$.

Solving~\eqref{eq:functional equation2} analytically is intractable
for most choices of the offspring distribution. 
We therefore resort to the perturbation method~\cite{Holmes2013,Hinch1991} to study the solution of~\eqref{eq:functional equation2} as $\lambda \downarrow 1$. 
Let 
\begin{align}
\epsilon = \lambda-1
\end{align}
be the perturbation parameter, 
and assume that, for a sufficiently small $\epsilon \geq 0$, the cumulant generating function $\psi(t)$ of the offspring distribution is expanded in $\epsilon$ as
\begin{align}
\psi(t) = \psi_0 (t) + \epsilon \psi_1 (t) + \mathcal{O}(\epsilon^2),
\label{eq:psi_expansion_epsilon}
\end{align}
where 
\begin{align}
\psi_0 (t) = \sum_{k=1}^{\infty} \frac{\kappa_k^*}{k!}t^k = t + \sum_{k=2}^{\infty} \frac{\kappa_k^*}{k!}t^k
\end{align}
is the cumulant generating function of the offspring distribution at $\epsilon = 0$ (i.e., at $\lambda = 1$),
where $\kappa_1^*=\mathbb{E}[X_n^{(i)}]=\lambda=1$ holds.
We also assume that $\kappa_k^*$ for $k\ge2$ are such that
the power series defining $\psi_0(t)$ is convergent in a neighborhood of $t=0$. 
On the other hand, for $\epsilon\ge0$ one should have
\begin{equation}
  \mathbb{E}[X_n^{(i)}]
  =1+\epsilon=\psi'(0)=\psi_0'(0)+\epsilon\psi_1'(0)+\mathcal{O}(\epsilon^2),
\end{equation}
which yields $\psi_1'(0)=1$. 
Then, \eqref{eq:psi_expansion_epsilon} can be rewritten as 
\begin{align}
\psi(t) = \psi_0 (t) + \epsilon \left(t + \tilde{\psi}_1 (t) \right) + \mathcal{O}(\epsilon^2),
\label{eq:psi_expansion_epsilon2}
\end{align}
where $\tilde{\psi}_1 (t) = \mathcal{O}(t^2)$ as $t \to 0$.

We next consider expansion of the cumulant generating function $\bar{K}(t)$ of $\bar{W}$ in $\epsilon$.
 When $\epsilon=0$, one has $\bar{W}=\epsilon W=0$ with probability 1, yielding $\bar{K}(t)\equiv0$ which satisfies~\eqref{eq:functional equation2} trivially. 
We can therefore assume the following perturbative expansion
\begin{align}
\bar{K}(t)=\epsilon \bar{K}_1(t) + \epsilon^2 \bar{K}_2(t) + \mathcal{O}(\epsilon^3),
\label{eq:barK_expansion_epsilon}
\end{align}
where $\bar{K}_1(t)$ and $\bar{K}_2(t)$ denote the first- and second-order coefficients, respectively.
We note that for any $\epsilon\ge0$ one has $\bar{K}'(0)=\mathbb{E}[\bar{W}]=\epsilon$,
which implies $\bar{K}_1'(0)=1$. 
Substituting~\eqref{eq:barK_expansion_epsilon} into the left-hand side of~\eqref{eq:functional equation2} yields, up to the second order in $\epsilon$,
\begin{align}
\bar{K}(\lambda t) &= \bar{K}\left((1+\epsilon) t \right) \nonumber \\
&= \bar{K}(t) + \epsilon t \bar{K}'(t) + \frac{\epsilon^2 t^2}{2}\bar{K}''(t) + \mathcal{O}(\epsilon^3) \nonumber \\
&= \epsilon \bar{K}_1(t) + \epsilon^2 \left( \bar{K}_2(t) + t \bar{K}'_1(t) \right) + \mathcal{O}(\epsilon^3).
\label{eq:perturb_left}
\end{align}
Using~\eqref{eq:psi_expansion_epsilon2} and~\eqref{eq:barK_expansion_epsilon}, the right-hand side of~\eqref{eq:functional equation2} is calculated, up to the second order in $\epsilon$, as
\begin{align}
\psi(\bar{K}(t)) &= \psi_0\left( \epsilon \bar{K}_1(t) + \epsilon^2 \bar{K}_2(t) + \mathcal{O}(\epsilon^3) \right) \nonumber \\
&~\quad+ \epsilon \left( \epsilon \bar{K}_1(t) + \mathcal{O}(\epsilon^2) \right) +  \mathcal{O}(\epsilon^3) \nonumber \\
&= \epsilon \bar{K}_1(t) + \epsilon^2 \left( \bar{K}_2(t) + \frac{\kappa_2^*}{2}\{\bar{K}_1(t)\}^2 + \bar{K}_1(t) \right) \nonumber \\
&~\quad+ \mathcal{O}(\epsilon^3).
\label{eq:perturb_right}
\end{align}
Equating \eqref{eq:perturb_left} and \eqref{eq:perturb_right}, the zeroth-order terms vanish on both sides, and the first-order terms coincide.
Comparing the coefficients of $\epsilon^2$ yields the following ordinary differential equation for $\bar{K}_1(t)$:
\begin{align}
t \bar{K}'_1(t) = \frac{\kappa_2^*}{2}\{\bar{K}_1(t)\}^2 + \bar{K}_1(t).
\label{eq:ODE_barK_1}
\end{align}
We solve it under the condition $\bar{K}_1'(0)=1$,
yielding
\begin{align}
\bar{K}_1(t) = \frac{t}{1-\frac{\kappa_2^* t}{2}}.
\label{eq:barK_1_solution}
\end{align}
Therefore, the perturbative approximation to the solution of the functional equation~\eqref{eq:functional equation2} as $\epsilon \downarrow 0$ (i.e., $\lambda \downarrow 1$), up to the first order in $\epsilon$, is obtained as 
\begin{align}
\bar{K}(t) &\approx \epsilon \bar{K}_1(t)  \nonumber \\
&=\frac{(\lambda-1)t}{1-\frac{\kappa_2^*t}{2}} =: \tilde{K}(t).
\label{eq:T_K solution}
\end{align}
It should be emphasized that the same result~\eqref{eq:T_K solution} can also be derived rigorously under weaker assumptions, without resorting to the perturbation method; the details are provided in Sec.~S1 of the Supplementary Material~\footnote{See Supplemental Material for the detailed discussion of the derivation of the approximate solution $\tilde{K}(t)$ without resorting to the perturbation method.}.

From the preceding argument, as $\lambda \downarrow 1$, the cumulant generating function of $W=\bar{W}/(\lambda-1)$ can be approximated as
\begin{equation}
  K_{W}(t)
  \approx \tilde{K}\left(\frac{t}{\lambda-1}\right)
  =\frac{t}{1-\frac{\kappa_2^*}{2(\lambda-1)}t}.
  \label{eq:cgf_W}
\end{equation}
A remarkable fact about the result is that the asymptotic cumulant generating function $\T{K}(t)$ is universal, in the sense that it depends on the offspring distribution only via its variance $\kappa_2^*$ in the limit $\lambda\downarrow1$. 

We next show that the cumulant generating function $\tilde{K}\left(\frac{t}{\lambda-1}\right)$ in~\NReq{cgf_W} coincides with that of the CPG distribution via an appropriate reparameterization.
A random variable $S$ following a CPG distribution is defined as 
\begin{align}
\Leq{CPG_def0}
S &= \sum_{i=1}^{N} Y_i, \\
N \sim {\rm Poisson}(\mu)&;~ Y_i \overset{\text{i.i.d.}}{\sim} {\rm Gamma}(\alpha, \tau),
\Leq{CPG_def}
\end{align}
where ${\rm Poisson}(\mu)$ denotes a Poisson distribution with mean $\mu$, and where ${\rm Gamma}(\alpha, \tau)$ denotes a gamma distribution with shape parameter $\alpha$ and scale parameter $\tau$.
Via a slight abuse of notation~\footnote{Strictly speaking, the CPG distribution does not have a probability density function, because it has the mass $e^{-\mu }>0$ at 0.},
we let $P_{\mathrm{CPG}}(\cdot;\mu,\alpha,\tau)$ denote
the probability density function of the CPG distribution
with parameters $(\mu,\alpha,\tau)$. 
The cumulant generating function $K_{S}(t)$ of $S$ is given by
\begin{align}
K_{S}(t) &= \log\left( \mathbb{E}_{S}[e^{tS}] \right) \no \\
&= \log\left( \mathbb{E}_{N}\left[ \left\{ \mathbb{E}_{Y_i}[e^{tY_i}] \right\}^{N} \right] \right) \no \\
&= \log\left( \mathbb{E}_{N}\left[ e^{NK_{Y_i}(t)} \right] \right) \no \\
&= K_{N}\left(K_{Y_i}(t)\right) \no \\
&= \mu \left\{ (1-\tau t)^{-\alpha} - 1 \right\},
\Leq{K_S}
\end{align}
where $K_{N}(t)$ and $K_{Y_i}(t)$ are the cumulant generating functions of $N$ and $Y_i$, given by
\begin{align}
K_{N}(t) &= \mu (e^t-1), \\
K_{Y_i}(t) &= \log (1-\tau t)^{-\alpha},
\end{align}
respectively.
Comparing \NReq{K_S} with~\eqref{eq:cgf_W}, we find that the cumulant generating function~\eqref{eq:cgf_W} coincides with that of the CPG distribution
$P_{\mathrm{CPG}}(\cdot;\mu,\alpha,\tau)$ with 
\begin{align}
\mu = \frac{2(\lambda-1)}{\kappa_2^*};~ \alpha = 1;~ \tau = \frac{\kappa_2^*}{2(\lambda-1)}.
\Leq{CPG param}
\end{align}

Consequently, it is shown that in the asymptotic regime $\lambda \downarrow 1$ the limiting distribution $\mathbb{P}(W)$ of the GW process on Condition I can be approximated by the CPG distribution with the parameters \NReq{CPG param}.


\subsection{Condition II: Random $Z_0$}
\Lsec{Condition II}

Motivated by practical applications, in this subsection, we consider the case where $Z_0$ follows a certain distribution $\mathbb{P}(Z_0)$ with mean $\lambda_0$ (Condition II).
Note that, unlike $\lambda$, $\lambda_0$ is not required to be close to unity.

Let us assume that the cumulant generating function $K_{Z_0}(t)$ of $Z_0$ exists:
\begin{align}
K_{Z_0}(t) = \sum_{\nu=1}^{\infty} \frac{\kappa_{Z_0, \nu}}{\nu!}t^\nu,
\end{align}
where $\kappa_{Z_0, \nu}$ denotes the $\nu$th cumulant of $Z_0$ and $\kappa_{Z_0, 1} = \lambda_0$.
Then, consider a random variable $V_n$ defined by
\begin{align}
V_n = \frac{Z_n}{\lambda^n}.
\Leq{def V_n}
\end{align}
Applying the result of the discussion in \Rsec{Condition I}, as $n \to \infty$ and $\lambda \downarrow 1$, one can approximately model the limiting distribution of $V_n$ as 
\begin{align}
V &= \sum_{i=1}^{Z_0} W^{(i)}, \Leq{T_W1}\\
Z_0 &\sim \mathbb{P}(Z_0),
\Leq{Z_0}
\end{align}
where $V$ denotes the limit of $V_n$ as $n \to \infty$.
$\{W^{(i)}\}$ denotes a set of i.i.d. random variables that follow the CPG distribution with parameters \NReq{CPG param}, that is, for any $i$, 
\begin{align}
W^{(i)} = \sum_{j=1}^{N^{(i)}} \Gamma_j^{(i)},
\end{align}
where
\begin{align}
N^{(i)} &\overset{\text{i.i.d.}}{\sim} {\rm Poisson}\left( \frac{2(\lambda-1)}{\kappa_2^*} \right), \Leq{N_i}\\
\Gamma_j^{(i)} &\overset{\text{i.i.d.}}{\sim} {\rm Gamma}\left(1, \frac{\kappa_2^*}{2(\lambda-1)} \right).
\end{align}

Here, let
\begin{align}
N_0 = \sum_{i=1}^{Z_0} N^{(i)}.
\Leq{N_0 def}
\end{align}
Then, \NReq{T_W1} can be rewritten using $N_0$ as
\begin{align}
V &= \sum_{j=1}^{N_0} \Gamma_j, \Leq{T_W2} \\
\Gamma_j &\overset{\text{i.i.d.}}{\sim} {\rm Gamma}\left(1, \frac{\kappa_2^*}{2(\lambda-1)} \right). \Leq{gamma_j}
\end{align}
From \NReq{N_0 def} and \NReq{N_i}, the cumulant generating function $K_{N_0}(t)$ of $N_0$ can be obtained in a similar manner to \NReq{K_S}, as follows:
\begin{align}
K_{N_0}(t) &= K_{Z_0}(K_{N^{(i)}}(t)) \no \\
&= K_{Z_0} \left( \frac{2(e^{t}-1)}{\kappa_2^*} (\lambda-1) \right),
\Leq{K_N_0}
\end{align}
where $K_{N^{(i)}}(t)$ is the cumulant generating function of the Poisson random variable $N^{(i)}$ \NReq{N_i}.
Noting our assumption $\lambda-1\ll1$, from the Taylor series expansion of $K_{N_0}(t)$ in the quantity $\frac{2(e^{t}-1)}{\kappa_2^*} (\lambda-1)$, one has, for any $t$, 
\begin{align}
\lim_{\lambda \downarrow 1} \frac{K_{N_0}(t)}{\lambda-1} = \frac{2 \lambda_0}{\kappa_2^*} (e^{t}-1).
\end{align}
Thus, as $\lambda \downarrow 1$, 
\begin{align}
K_{N_0}(t) = \frac{2 \lambda_0 (\lambda-1)}{\kappa_2^*} (e^{t}-1),
\end{align}
which is nothing but the cumulant generating function
of the Poisson distribution with mean $\frac{2 \lambda_0 (\lambda-1)}{\kappa_2^*}$.
Therefore, one has
\begin{align}
N_0 \sim \text{Poisson}\left(\frac{2 \lambda_0 (\lambda-1)}{\kappa_2^*} \right)
\Leq{N_0 approx model}
\end{align}
in the limit $\lambda \downarrow 1$.
From \NReq{T_W2}, \NReq{gamma_j}, and \NReq{N_0 approx model}, it is shown that as $\lambda \downarrow 1$ the limiting distribution $\mathbb{P}(V)$ can be approximated by the CPG distribution 
$P_{\mathrm{CPG}}(\cdot;\mu,\alpha,\tau)$ with 
\begin{align}
\mu = \frac{2 \lambda_0 (\lambda-1)}{\kappa_2^*};~ \alpha = 1;~ \tau = \frac{\kappa_2^*}{2(\lambda-1)}.
\Leq{CPG param II}
\end{align}
It should be noted that the above CPG distribution
approximating $\mathbb{P}(V)$ depends on the distribution of $Z_0$
only via its mean $\lambda_0$. 

\section{Numerical Verification}\Lsec{Numerical Verification}

\subsection{Overview and Computational Issues}
Here, we examined the above theoretical argument by comparing numerically evaluated distributions of $Z_n$ for large $n$ with the CPG distributions. 
At the origin, since both $Z_n$ and the CPG random variable have probability mass, they can be compared as
\begin{align}
\mathbb{P}(Z_{n}=0) &\approx P_{\rm CPG}\lb 0 ;\mu, \alpha, \tau \rb
\no \\
&= e^{-\mu},
\Leq{CPG_comparison_at_origin}
\end{align}
where $(\mu, \alpha, \tau)$ are given by \eqref{eq:CPG param} under Condition I and by \eqref{eq:CPG param II} under Condition II.
On the other hand, on the interval $(0, \infty)$, a direct comparison of these distributions as above is not possible because the CPG random variable (and $W$) is absolutely continuous, whereas $Z_n$ is integer-valued and hence discrete.
Thus, we compare them in terms of probability values via appropriate scaling as follows:
\begin{align}
\mathbb{P}(Z_{n}=\ell) \approx 
\frac{1}{\lambda^n} P_{\rm CPG}\lb \frac{\ell}{ \lambda^n} ;\mu, \alpha, \tau \rb
\end{align}
for $\ell \in \mathbb{N}$, or equivalently,
\begin{align}
  P_{\rm CPG}\lb x;\mu,\alpha,\tau\rb\approx\lambda^n\mathbb{P}(Z_n=\lambda^nx)
\Leq{CPG_comparison}
\end{align}
for $x$ such that $\lambda^nx\in\mathbb{N}$, 
where $(\mu, \alpha, \tau)$ are given by \eqref{eq:CPG param} under Condition I and by \eqref{eq:CPG param II} under Condition II.
We performed quantitative verifications of these relations \Req{CPG_comparison_at_origin} and \Req{CPG_comparison} in this section.
We investigated both the Poisson and geometric offspring distributions
in order to demonstrate the universality of our results with respect to the choice of offspring distribution. 

For the Poisson offspring distribution, an efficient recursion formula to compute the distribution of $Z_n$ is known~\cite{Lombard1961, Dietz1965} and we employed it for computing $\mathbb{P}(Z_{n}=\ell)$.
For the sake of self-containedness, we provided the derivation of the recursion formula for Condition I in Appendix~\ref{sec:Recursion in the condition I} and for Condition II in Appendix~\ref{sec:Recursion in the condition II}, where we extended it to include certain negative binomial offspring distributions.
When computing $\mathbb{P}(Z_{n}=\ell)$ with the recursion formulas, the computational cost rapidly increases as $n$ grows, which imposes a limitation on the practically manageable support size of the distribution. Unless otherwise specified, we considered the range $\ell \leq \ell_{\mathrm{max}}$ with $\ell_{\mathrm{max}}=10^6$ for the Poisson case. Related to this computational issue, it should be noted that the maximum value of $n$ for which the simulation remains sufficiently accurate heavily depends on both the support size and the value of $\lambda$. In the results below, we show only the range of $n$ for which $\sum_{\ell=0}^{\ell_{\rm max}} \mathbb{P}(Z_n=\ell) > 0.99$ is satisfied. 

Meanwhile, for the geometric offspring distribution, the evaluation of $\mathbb{P}(Z_{n}=\ell)$ is much easier, since the probability generating function $f_n$ takes a linear fractional form for any generation $n$, allowing its parameters to be computed in a simple iterative manner; the distribution of $Z_n$ is calculated via a series expansion of the linear fractional form with the computed parameters~\cite{Sevastyanov1971}. An outline of the calculations for this case is provided in Appendix \ref{sec:Recursion Specific to the Geometric Distribution}. Thanks to this method, the computational difficulty encountered in the Poisson case is significantly alleviated in the geometric case.

\subsection{Results Under Condition I} 
We let $P_n(\ell)$ denote the probability $\mathbb{P}(Z_n=\ell)$
under Condition I.
The Poisson and geometric offspring distributions were considered. 

The results for the Poisson offspring distribution with $\lambda \in \{1.1, 1.5, 5.0\}$ are shown in \Rfig{Poisson_condI}.
\begin{figure*}[htbp]
\begin{center}
\includegraphics[width=0.67\columnwidth]{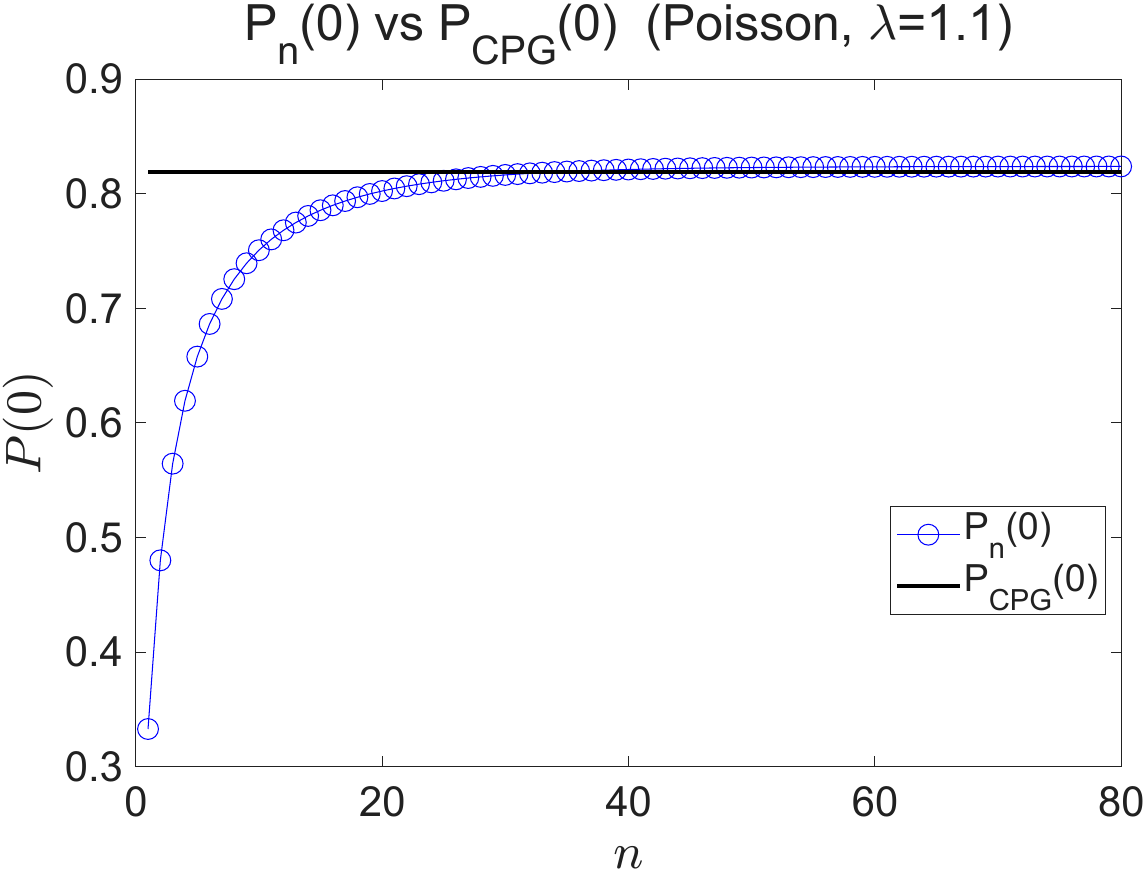}
\includegraphics[width=0.67\columnwidth]{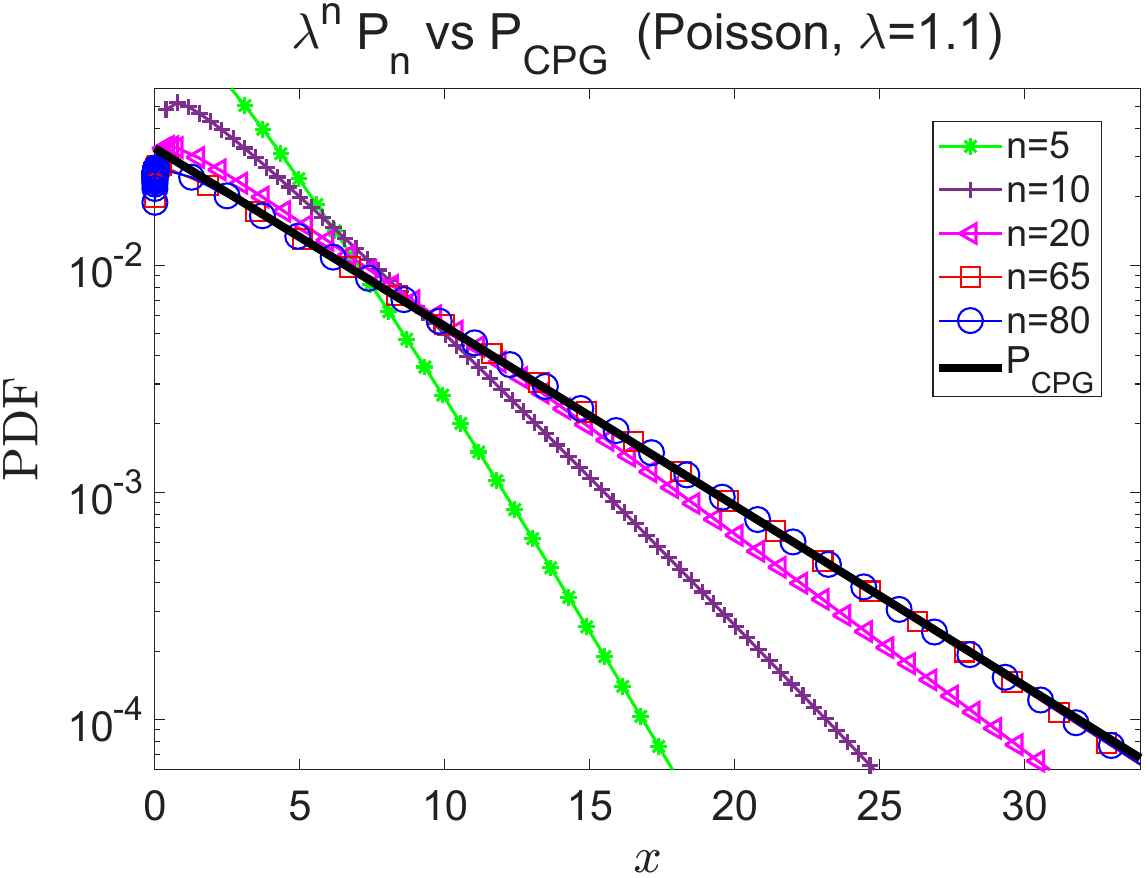}
\includegraphics[width=0.67\columnwidth]{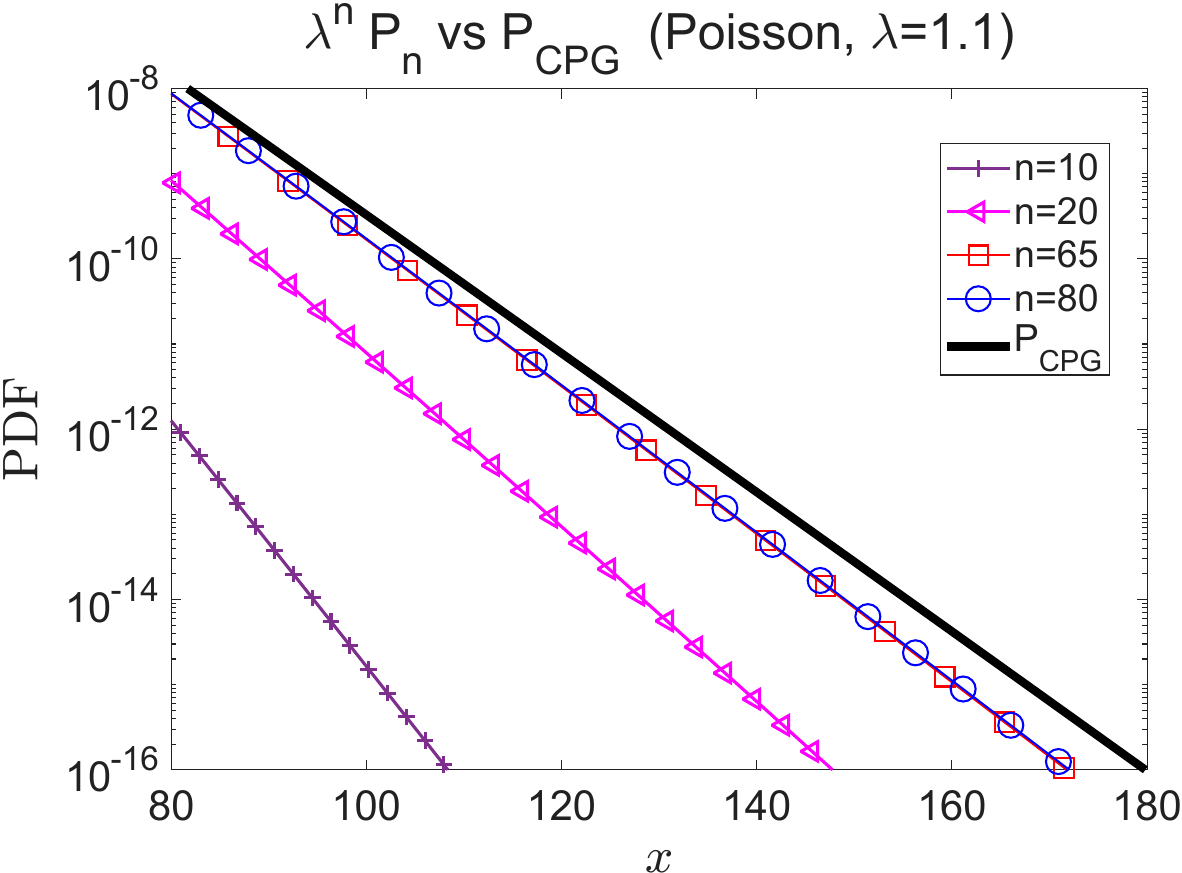}
\includegraphics[width=0.67\columnwidth]{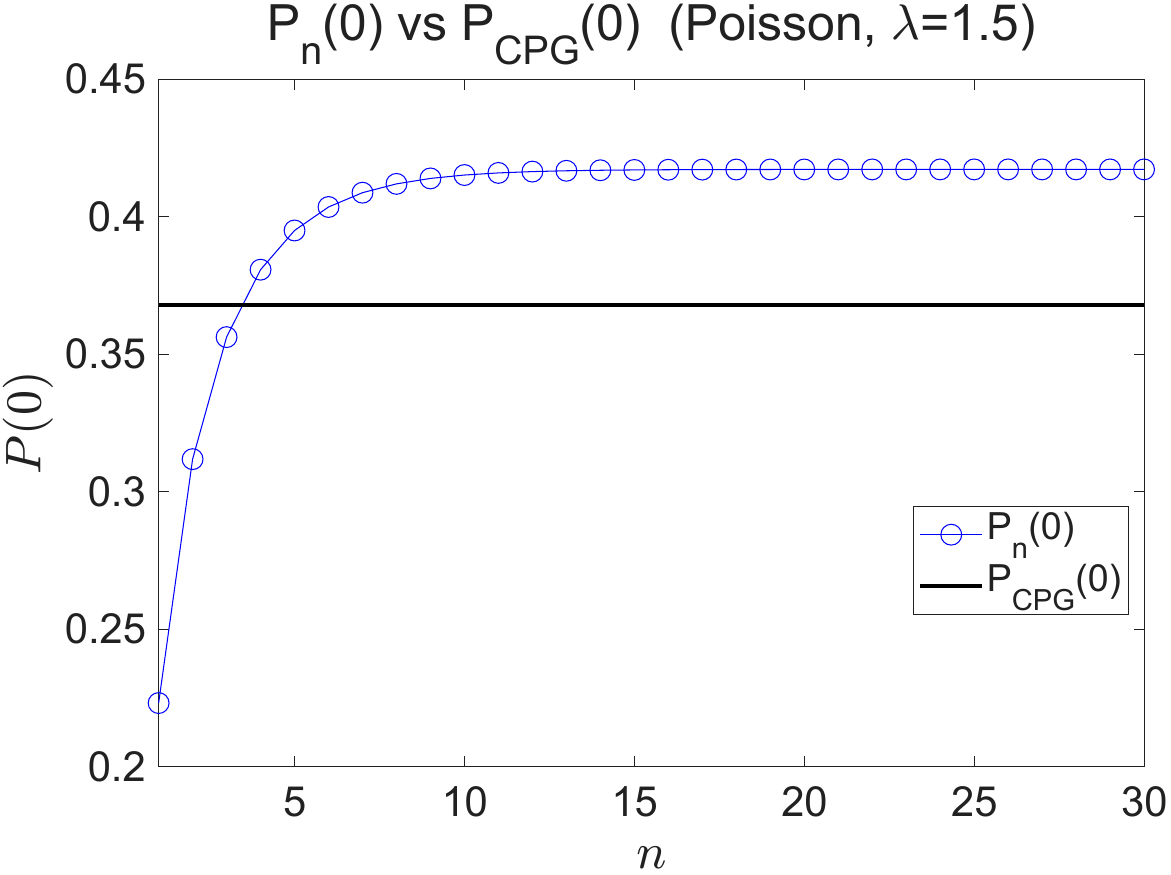}
\includegraphics[width=0.67\columnwidth]{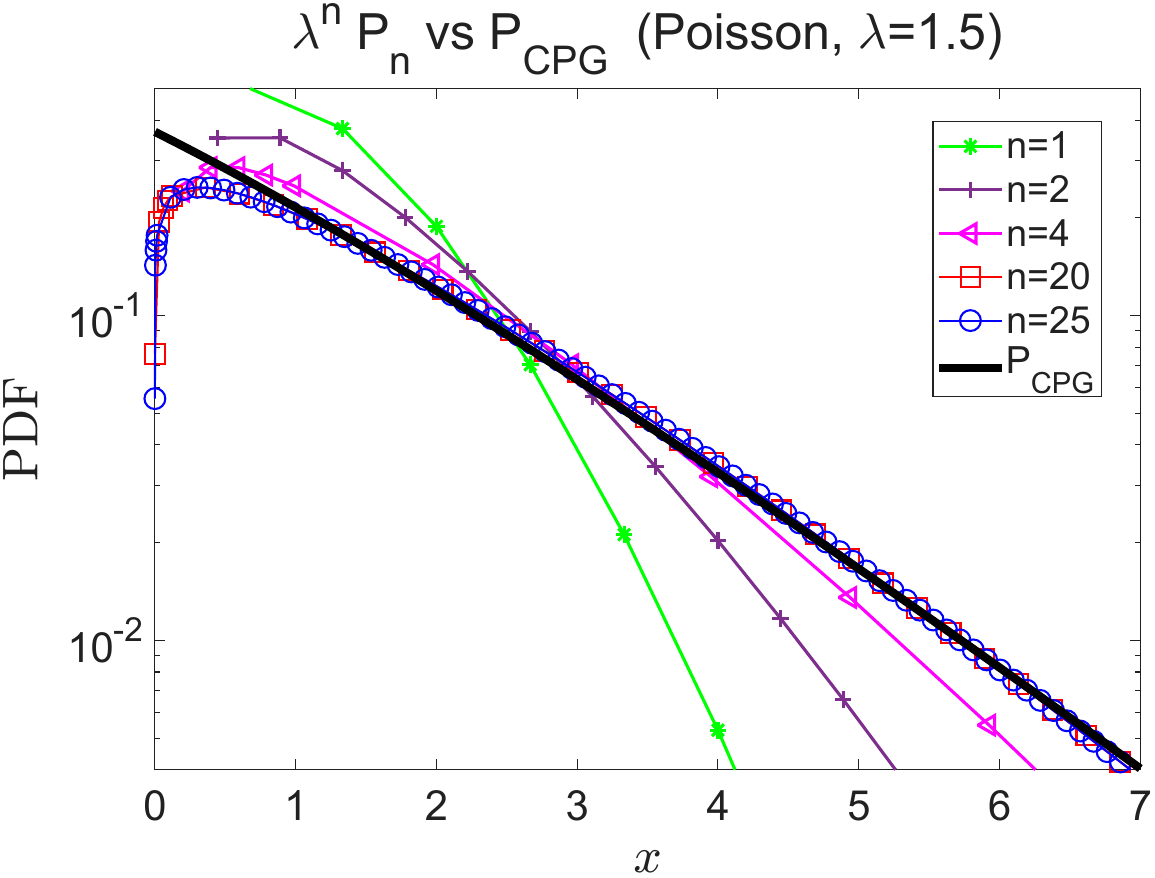}
\includegraphics[width=0.67\columnwidth]{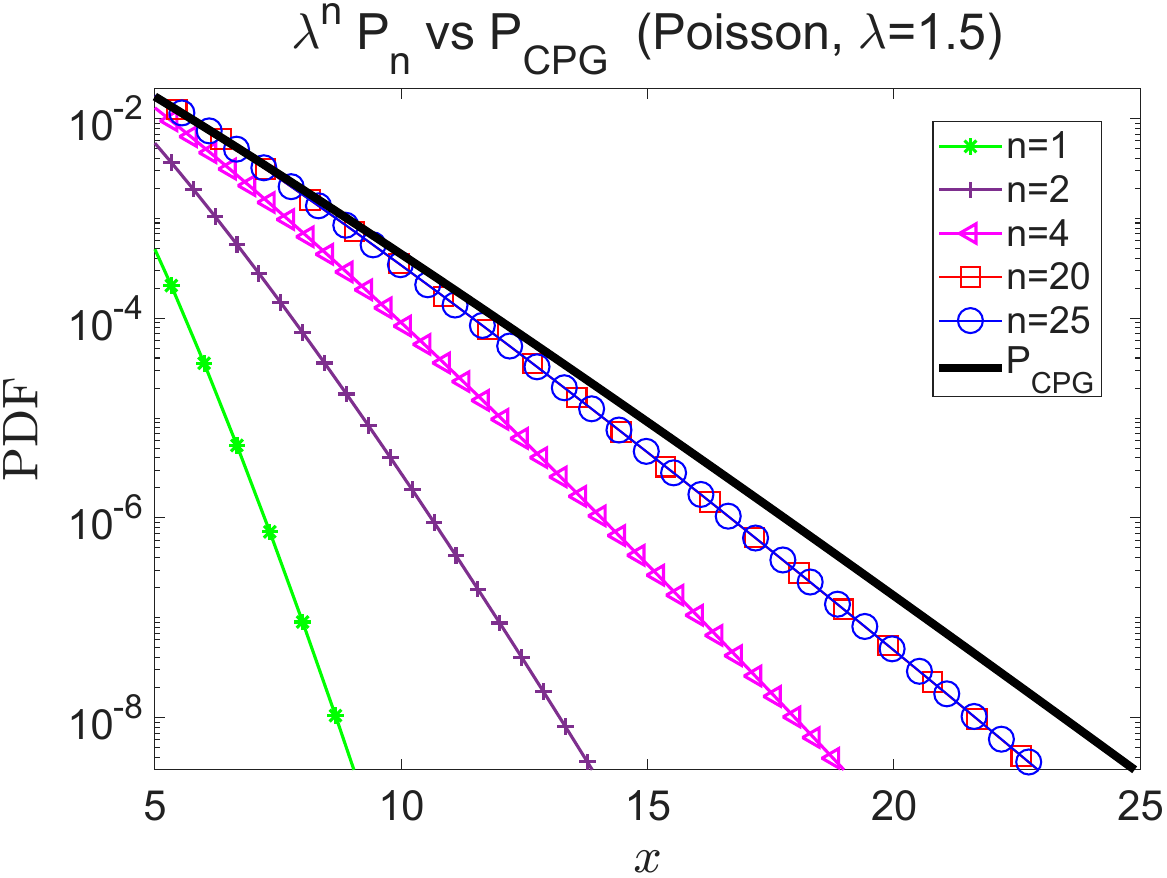}
\includegraphics[width=0.67\columnwidth]{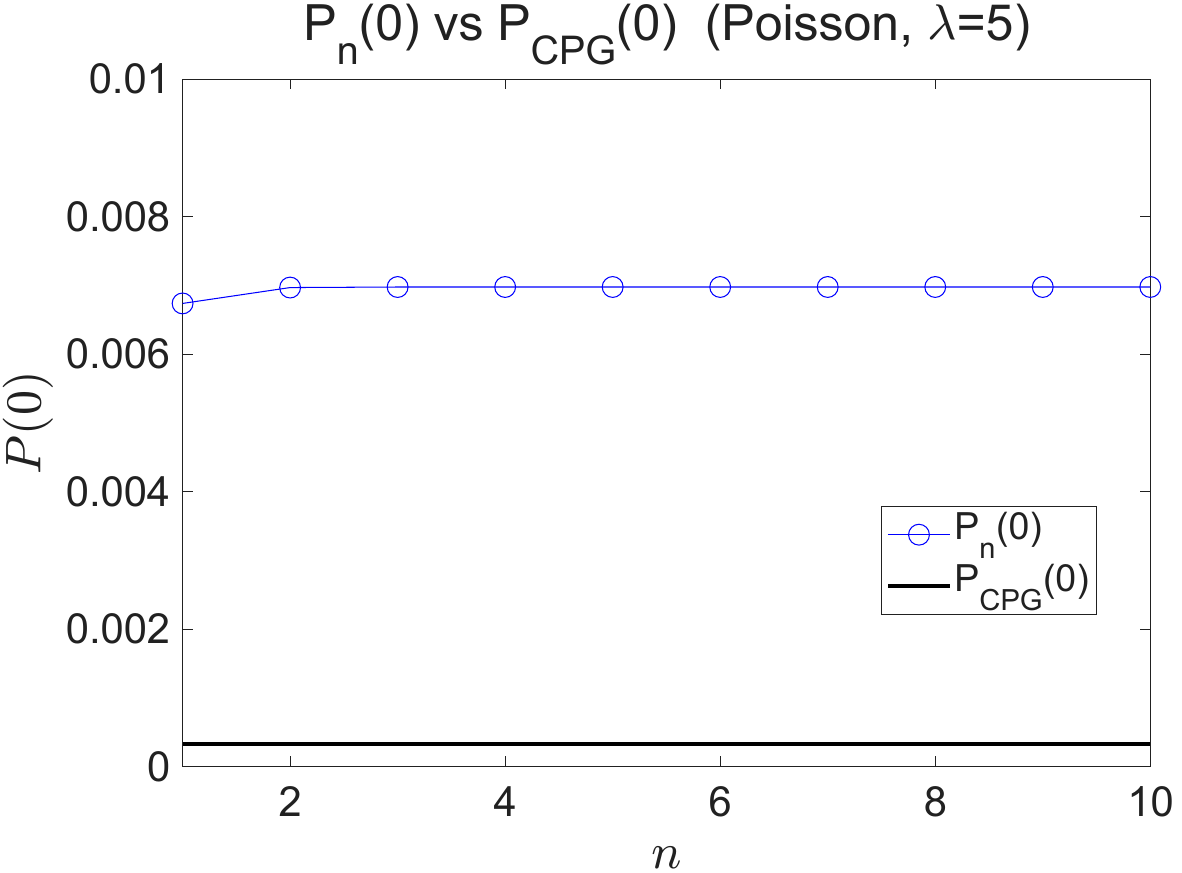}
\includegraphics[width=0.67\columnwidth]{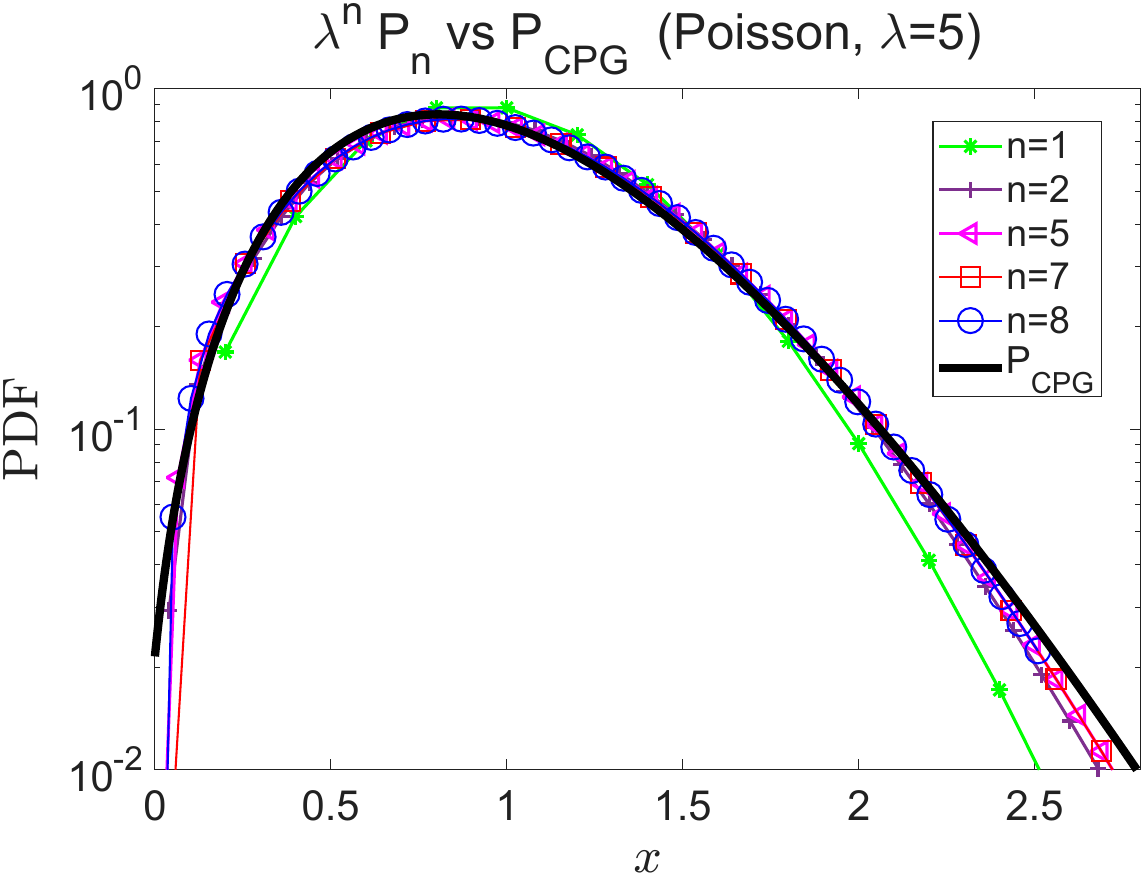}
\includegraphics[width=0.67\columnwidth]{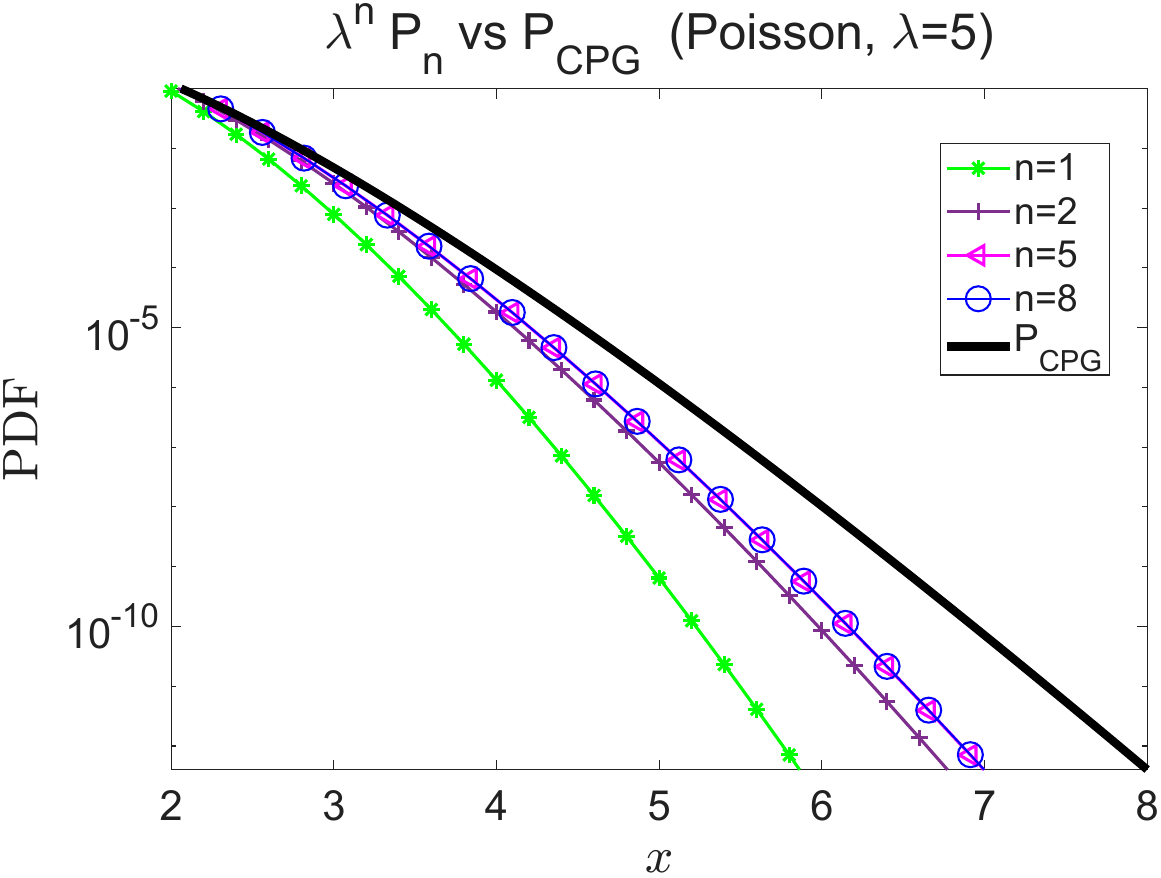}
\caption{Comparison between $P_{\mathrm{CPG}}$ and $P_n$ in the Poisson case under Condition I. Three different values of $\lambda$, $1.1,1.5,5.0$, are examined from top to bottom. The left column is for the plot of $P_n(0)$ (joined blue circles) against $n$, compared with $P_{\mathrm{CPG}}(0)$ (black line). The other columns compare $\lambda^n P_n(\lambda^n x)$ and $P_{\mathrm{CPG}}(x)$ in the semi-logarithmic scale. The bulk region and the right tail of the distribution are exhibited in the center and right columns, respectively. } 
\Lfig{Poisson_condI}
\end{center}
\end{figure*}
The top row shows results with $\lambda=1.1$, that is, for the case where $\lambda$ is close to the critical value 1.
One can observe that, as $n$ grows, $P_n(0)$ closely approaches $P_{\mathrm{CPG}}(0)$ (left panel) and that $P_{\mathrm{CPG}}(x)$ shows good agreement with $\lambda^n P_n(\lambda^n x)$ over a fairly wide range of $x$ for large $n$ (center panel).
These results are consistent with the theoretical argument above. 
Meanwhile, the right tail of $\lambda^n P_n(\lambda^n x)$ decays faster than the exponential decay of $P_{\mathrm{CPG}}(x)$ (right panel). 
This is natural, as will be discussed in Section~\ref{sec:tail}, because the limiting random variable $W$ of the GW process with an offspring distribution having a lighter-than-exponential tail, such as the Poisson distribution, should also have a lighter-than-exponential tail. Although the tail behaviors differ, the bulk region with non-negligible probability shows good agreement, supporting the validity of our approximation approach. 

The middle and bottom rows of \Rfig{Poisson_condI} show results for the Poisson offspring distribution with $\lambda=1.5$ and $5.0$, respectively. In both cases, the bulk regions of $P_{\rm CPG}(x)$ still exhibit fairly good agreement with $\lambda^n P_n(\lambda^n x)$, as shown in the center column. 
However, from the left column of \Rfig{Poisson_condI}, it was found that the CPG model underestimates the probability of zero of the GW model for large $n$.
We theoretically confirmed that this underestimation problem occurs for any $\lambda > 1$ (see Sec. S2 of the Supplemental Material~\footnote{See Supplemental Material at for a proof that the CPG approximation underestimates the probability of zero of the GW model as $n \to \infty$.}).

The results for the geometric offspring distribution with $\lambda \in \{1.1, 1.5, 5.0\}$ are shown in \Rfig{Geometric_condI}.
\begin{figure*}[htbp]
\begin{center}
\includegraphics[width=0.67\columnwidth]{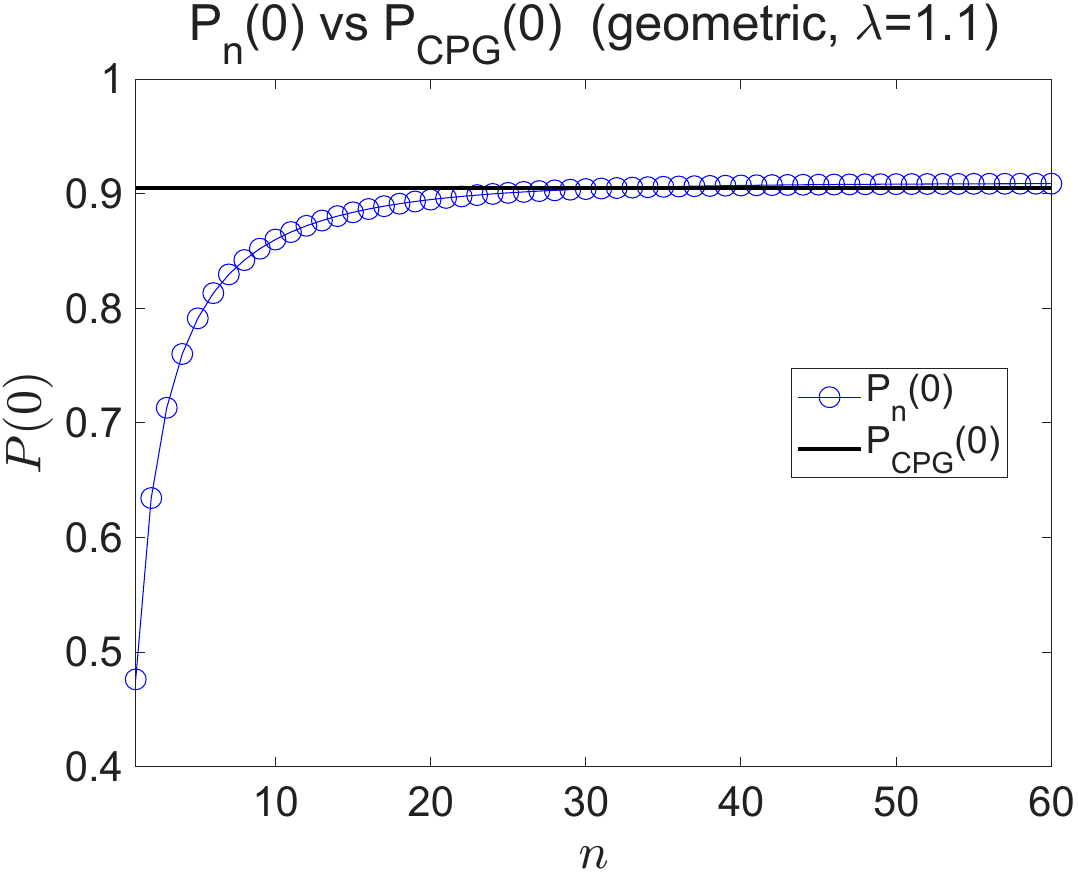}
\includegraphics[width=0.67\columnwidth]{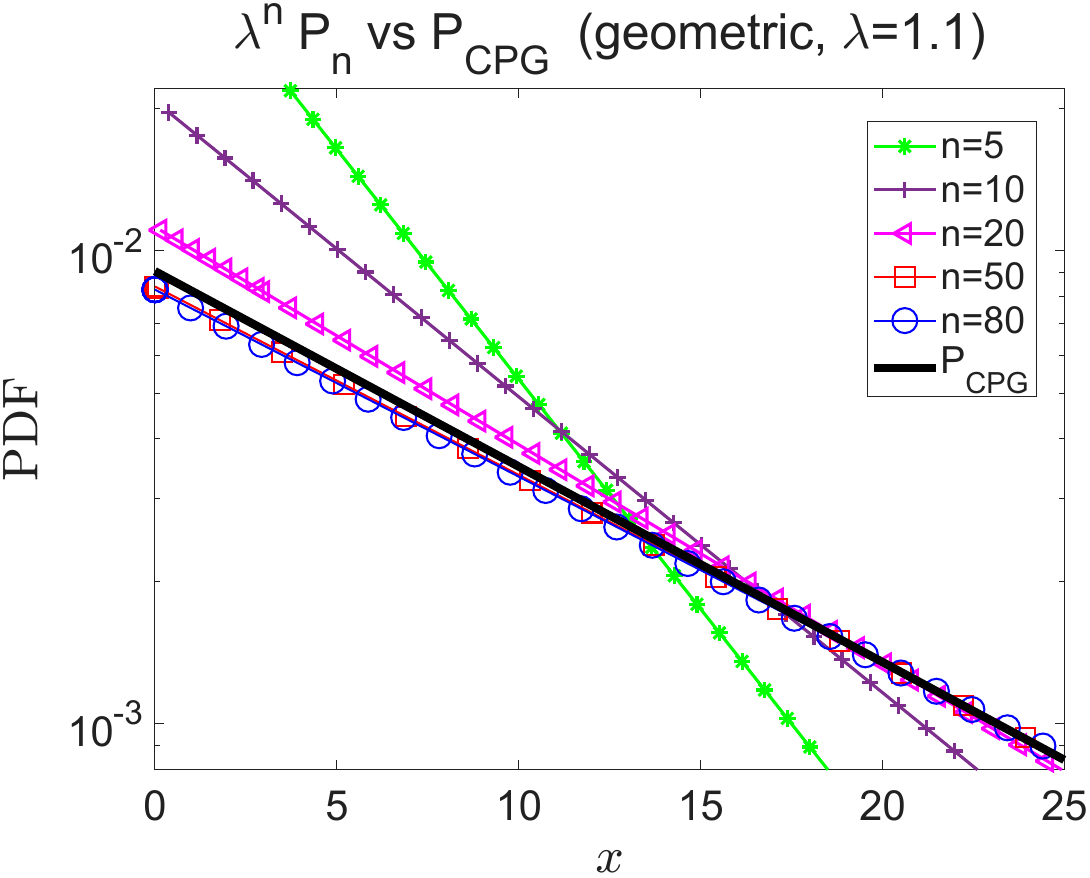}
\includegraphics[width=0.67\columnwidth]{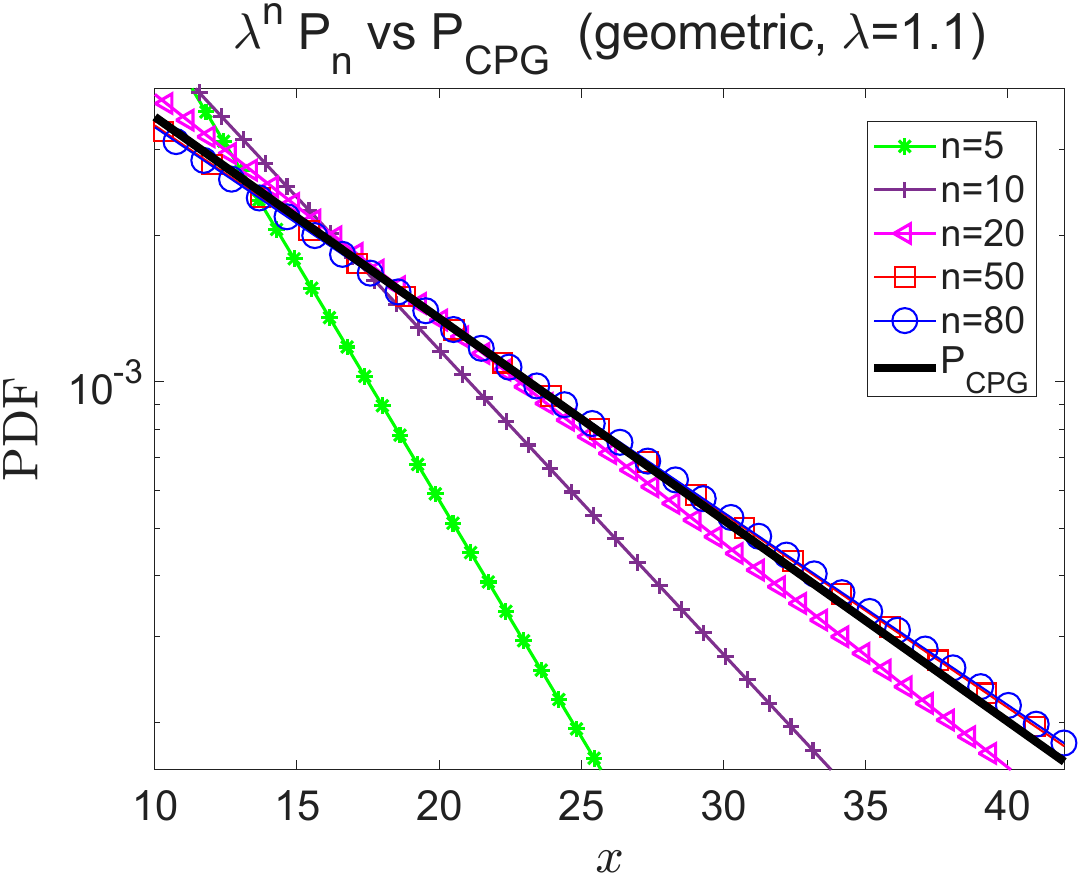}
\includegraphics[width=0.67\columnwidth]{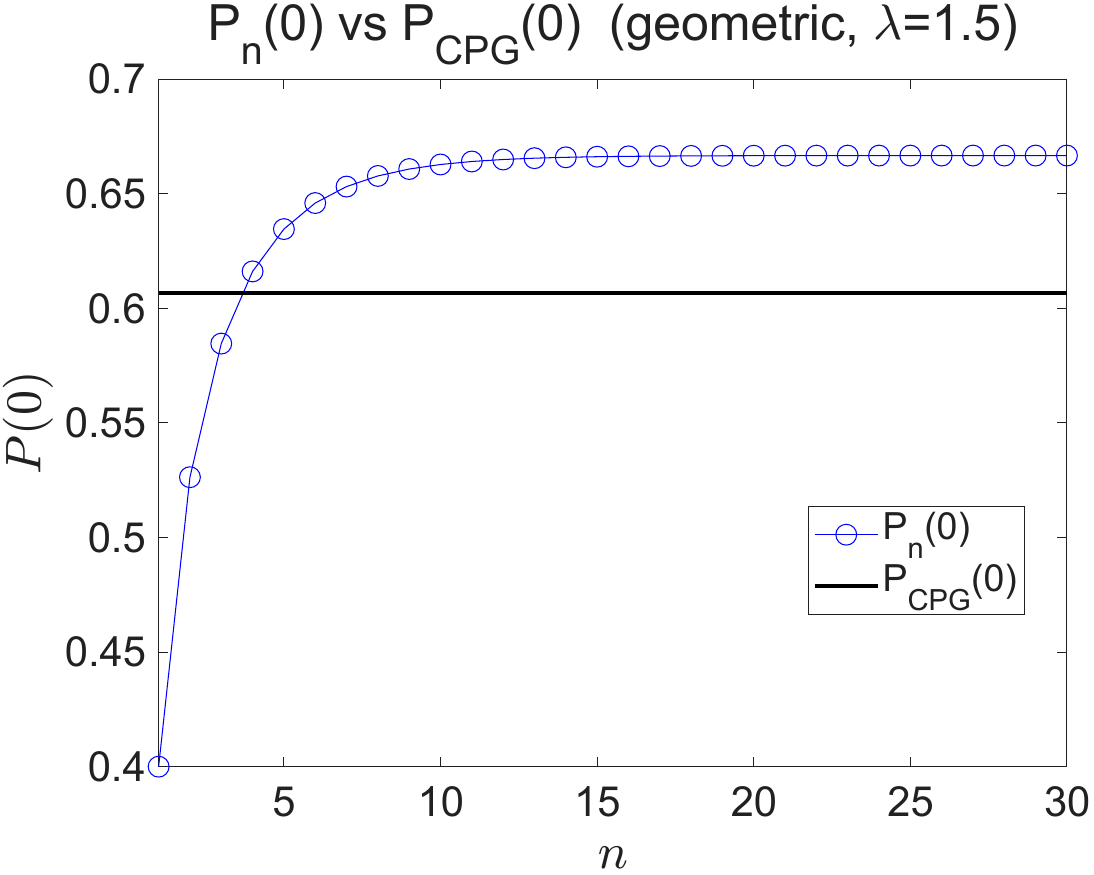}
\includegraphics[width=0.67\columnwidth]{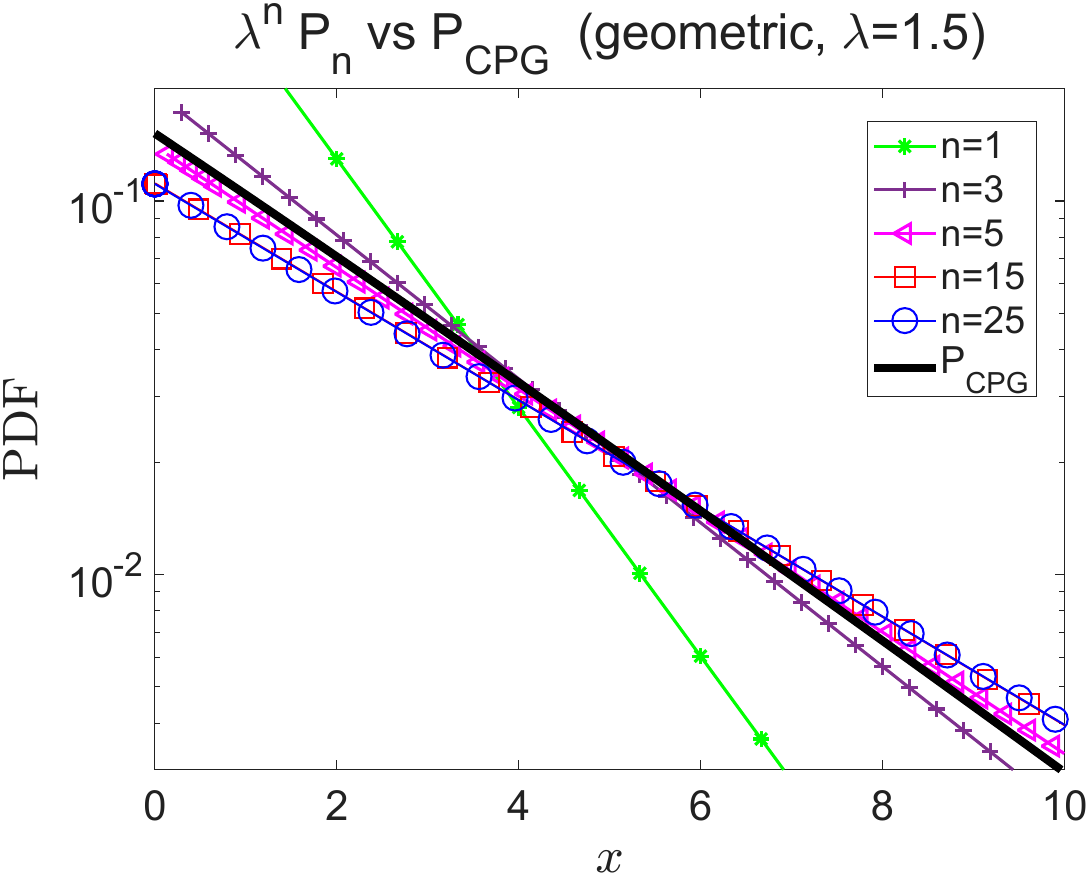}
\includegraphics[width=0.67\columnwidth]{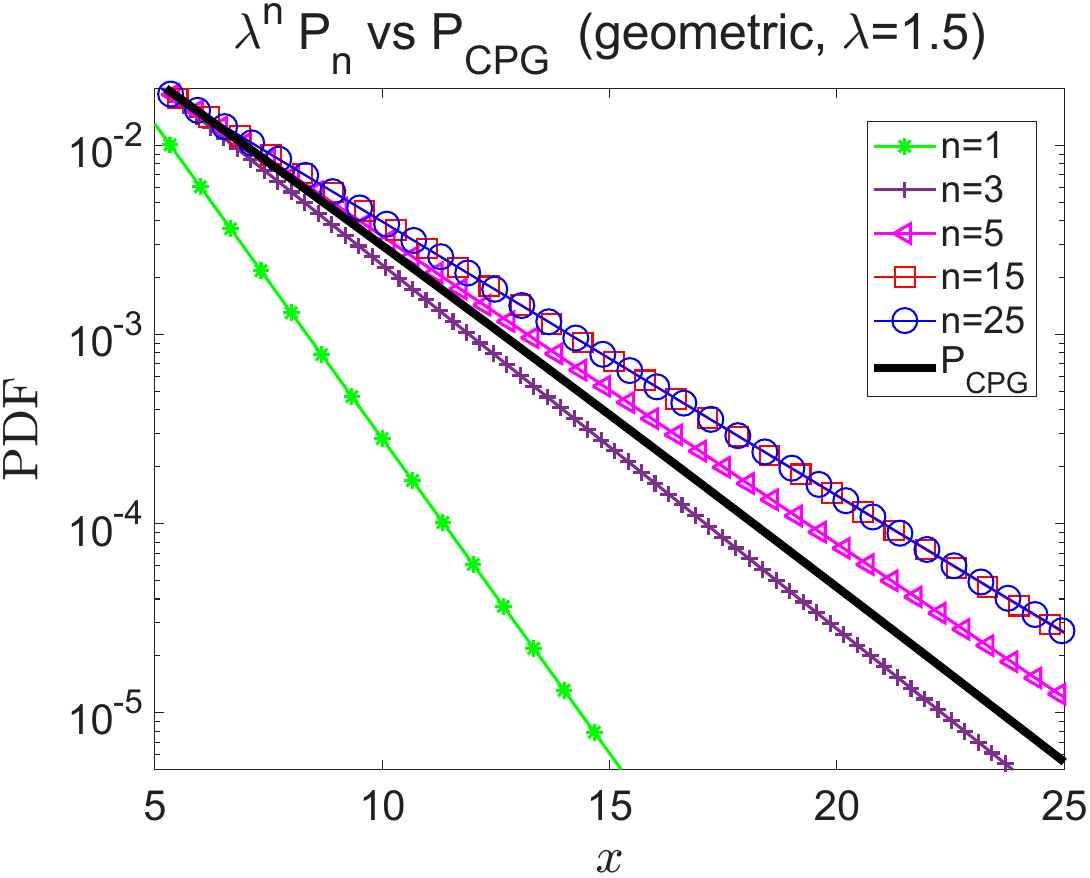}
\includegraphics[width=0.67\columnwidth]{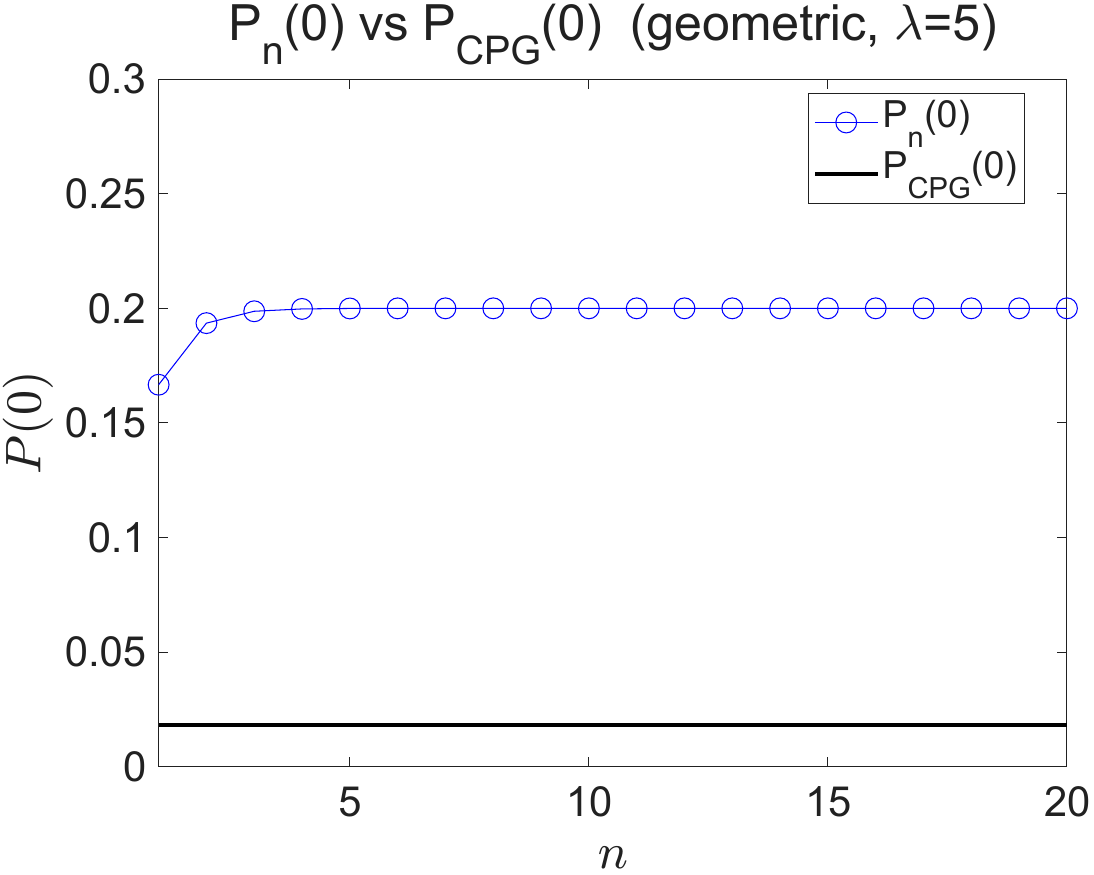}
\includegraphics[width=0.67\columnwidth]{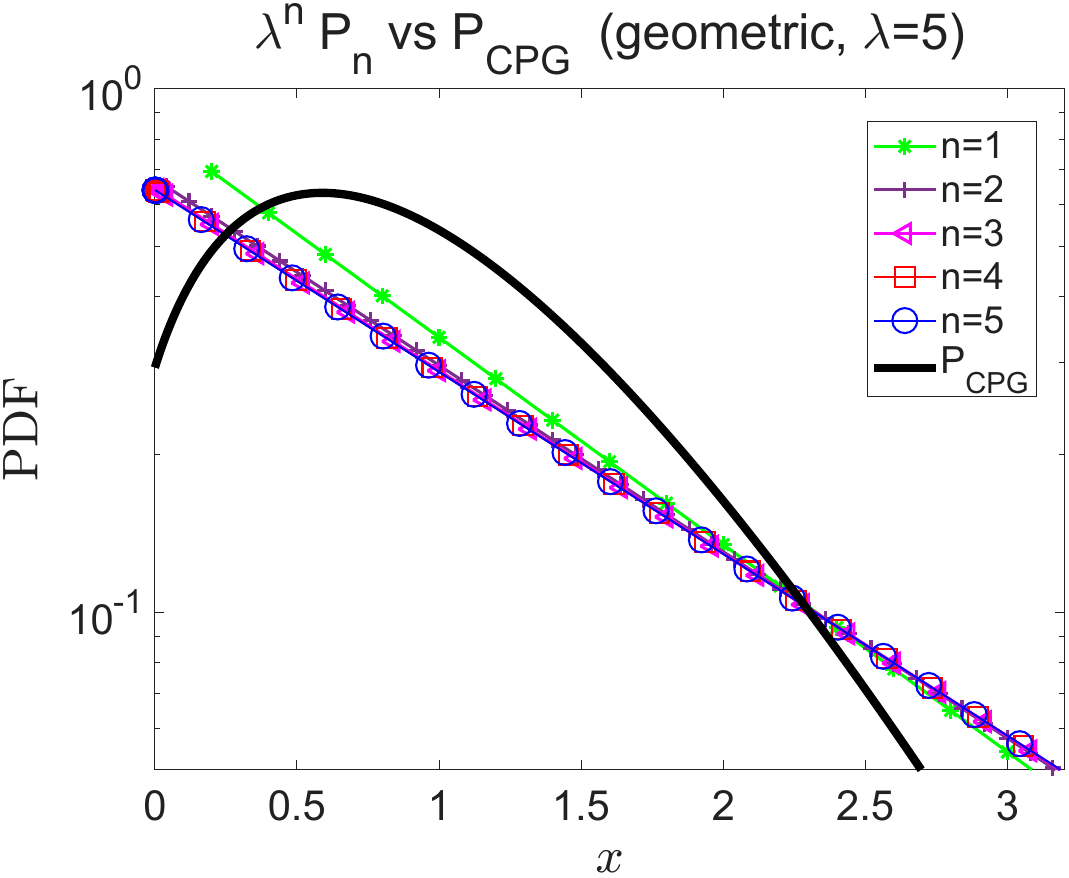}
\includegraphics[width=0.67\columnwidth]{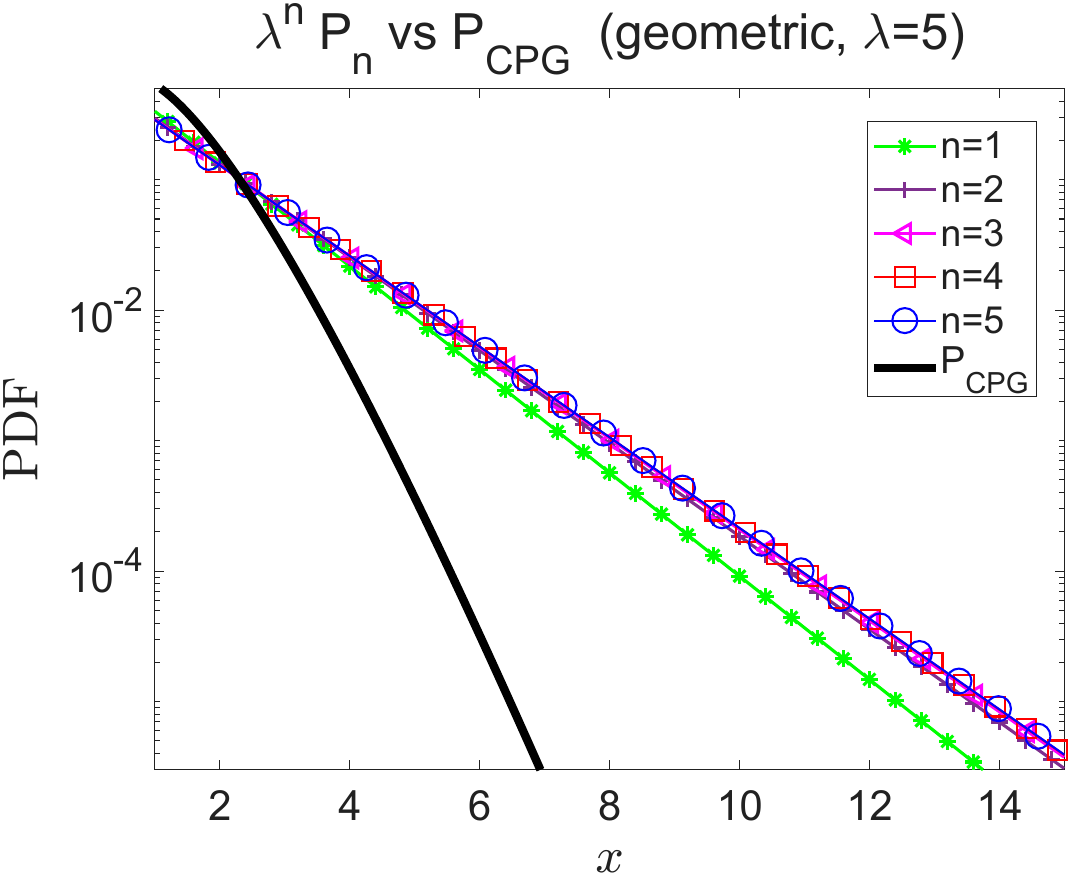}
\caption{Comparison between $P_{\mathrm{CPG}}$ and $P_n$ in the geometric case under Condition I. Three different values of $\lambda$, $1.1,1.5,5.0$, are examined from top to bottom. The left column is for the plot of $P_n(0)$ (joined blue circles) against $n$, compared with $P_{\mathrm{CPG}}(0)$ (black line). The other columns compare $\lambda^n P_n(\lambda^n x)$ (joined color markers) and $P_{\mathrm{CPG}}(x)$ (black line) in the semi-logarithmic scale. The bulk region and the right tail of the distribution are exhibited in the center and right columns, respectively. } 
\Lfig{Geometric_condI}
\end{center}
\end{figure*}
It is exactly the counterpart of \Rfig{Poisson_condI} and the corresponding panel plots the same quantity at the same parameter as that in \Rfig{Poisson_condI}: only the difference is in its offspring distribution. The top row again shows good agreement between $P_{\mathrm{CPG}}(x)$ and $\lambda^nP_n(\lambda^nx)$ at large $n$, demonstrating the universality of our theory when $\lambda$ is close to 1. On the other hand, 
the tail part of the distributions exhibits an exponential decay, in contrast to the Poisson case. This is because the geometric distribution has an exponential tail: the limiting random variable $W$ of the GW process with an offspring distribution having an exponential tail has an exponential tail as well (see Section~\ref{sec:tail} for a detailed discussion). Another intriguing observation, as shown in the bottom panels, is that for large $\lambda$, $P_{\mathrm{CPG}}(x)$ no longer approximates $\lambda^n P_n(\lambda^n x)$ at all.
This implies that, although the CPG distribution remains a relatively good approximation for the Poisson offspring distribution even for large $\lambda$, the same does not hold for the geometric offspring distribution. 

\subsection{Results Under Condition II}
We conducted numerical experiments under Condition II, assuming that $Z_0$ and $X_n^{(i)}$ in \NReq{GW process} follow Poisson distributions with means $\lambda_0$ and $\lambda$, respectively.
Although this setting is the simplest, it remains nontrivial; it has been adopted as a reasonable model for SER distributions in earlier studies~\cite{Dietz1965, Gale1966}.
Note that analogous calculations for the geometric offspring distribution can be performed by appropriately modifying the procedure described in Appendix~\ref{sec:Recursion Specific to the Geometric Distribution}.

We let $Q_n(\ell)$ denote the probability $\mathbb{P}(Z_n=\ell)$ under Condition II. 
We first examined cases where the mean $\lambda$ of the offspring distribution is close to unity.
Figure~\ref{fig:condII_close1} shows results for $\lambda_0=1,2,5$ at $\lambda=1.1$. 
\begin{figure*}[htbp]
\begin{center}
\includegraphics[width=0.67\columnwidth]{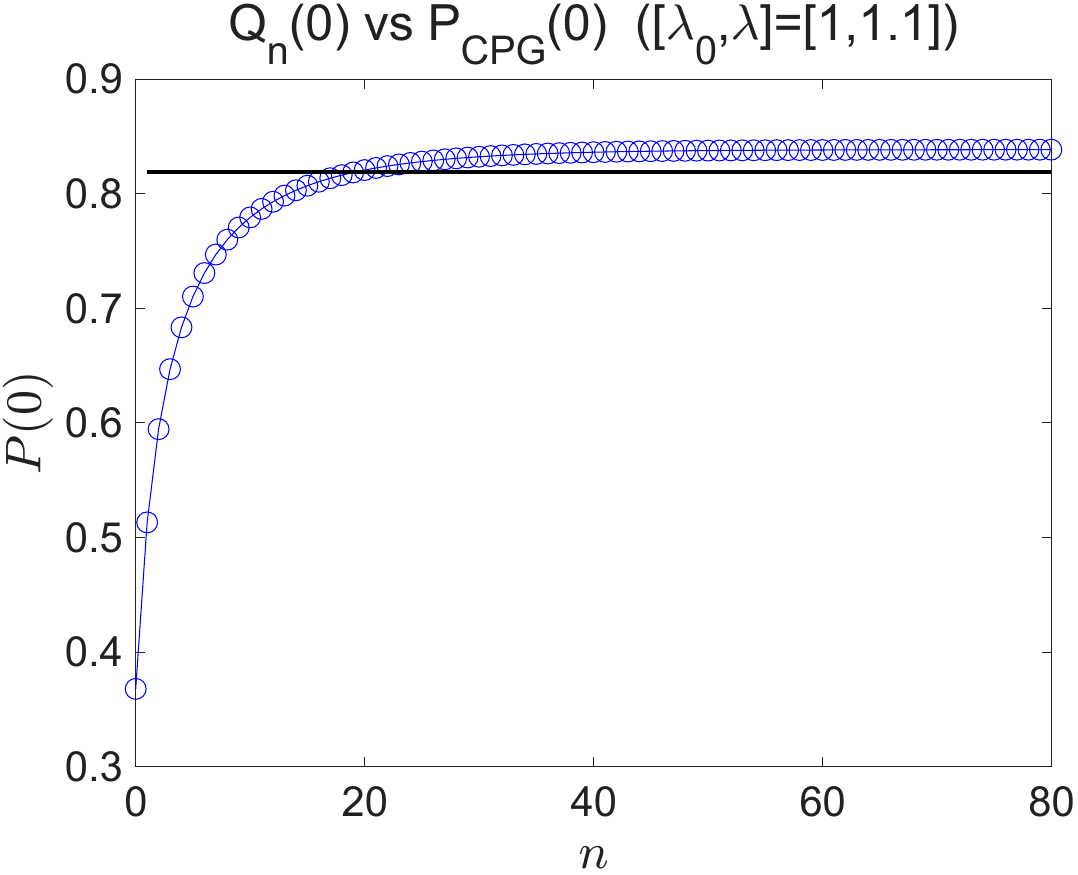}
\includegraphics[width=0.67\columnwidth]{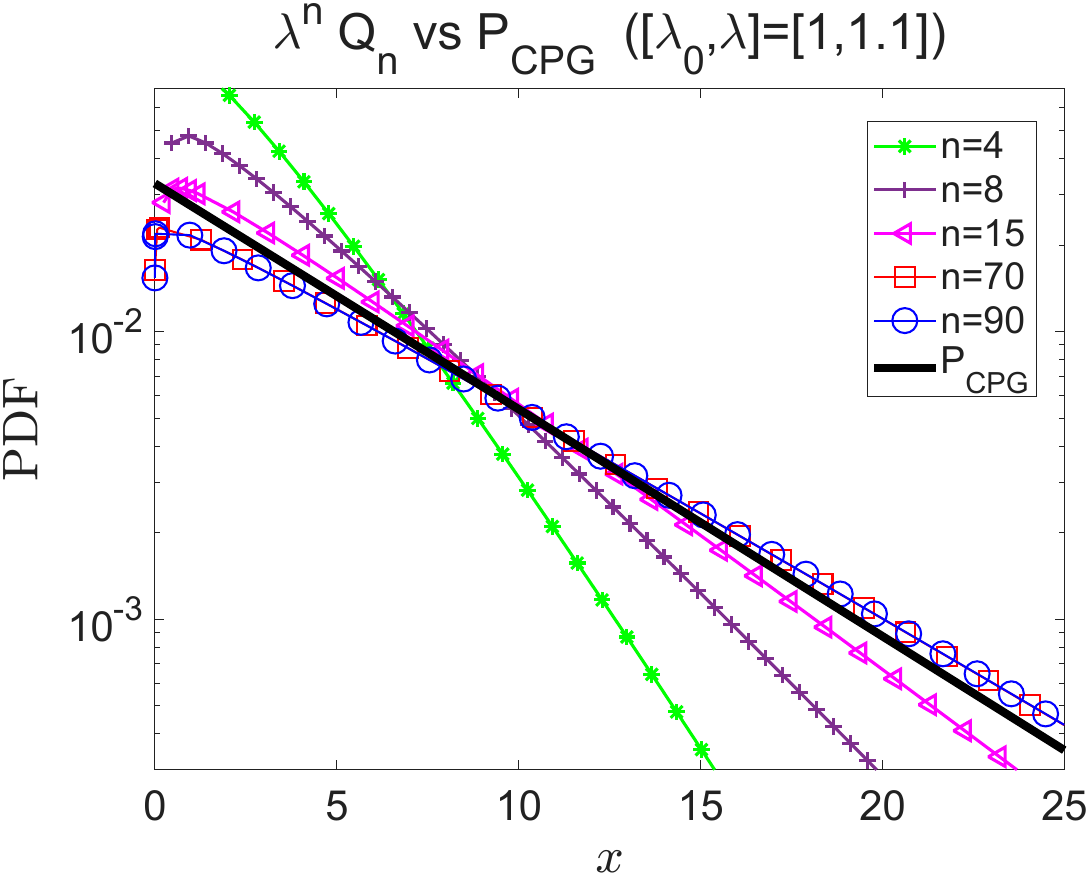}
\includegraphics[width=0.67\columnwidth]{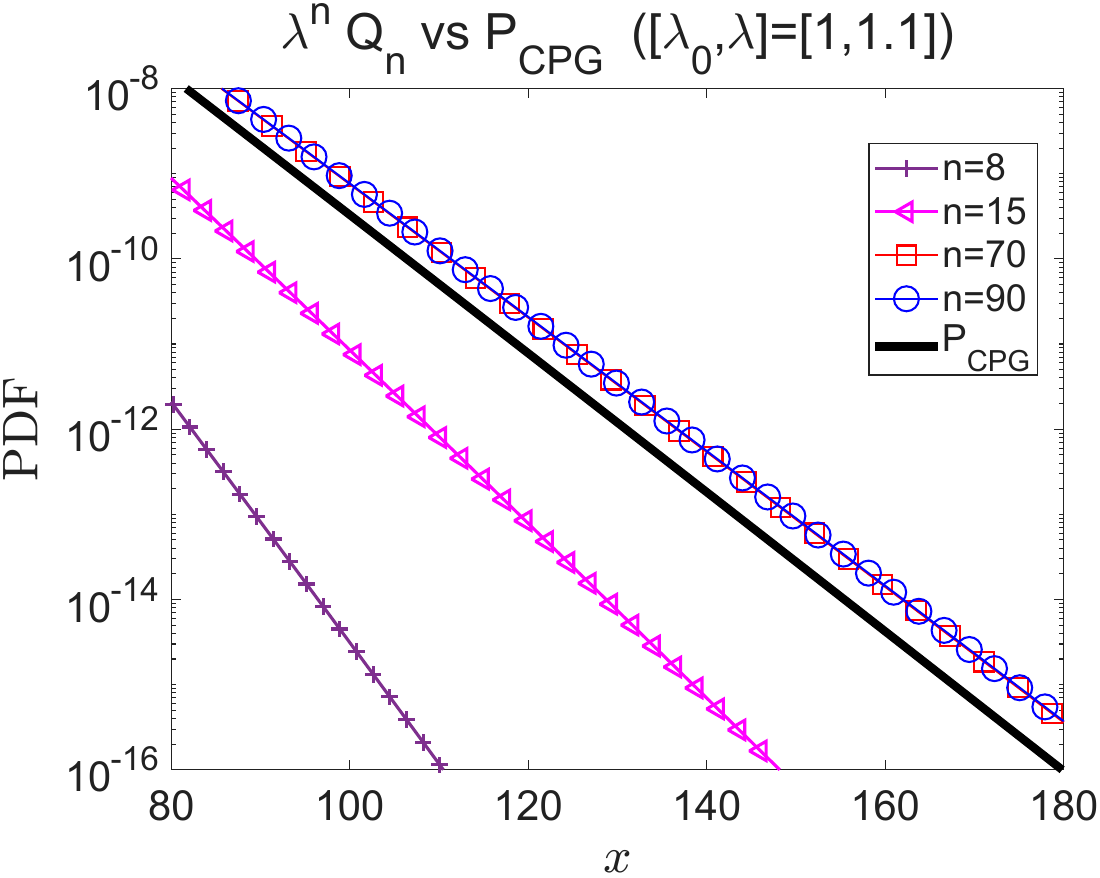}
\includegraphics[width=0.67\columnwidth]{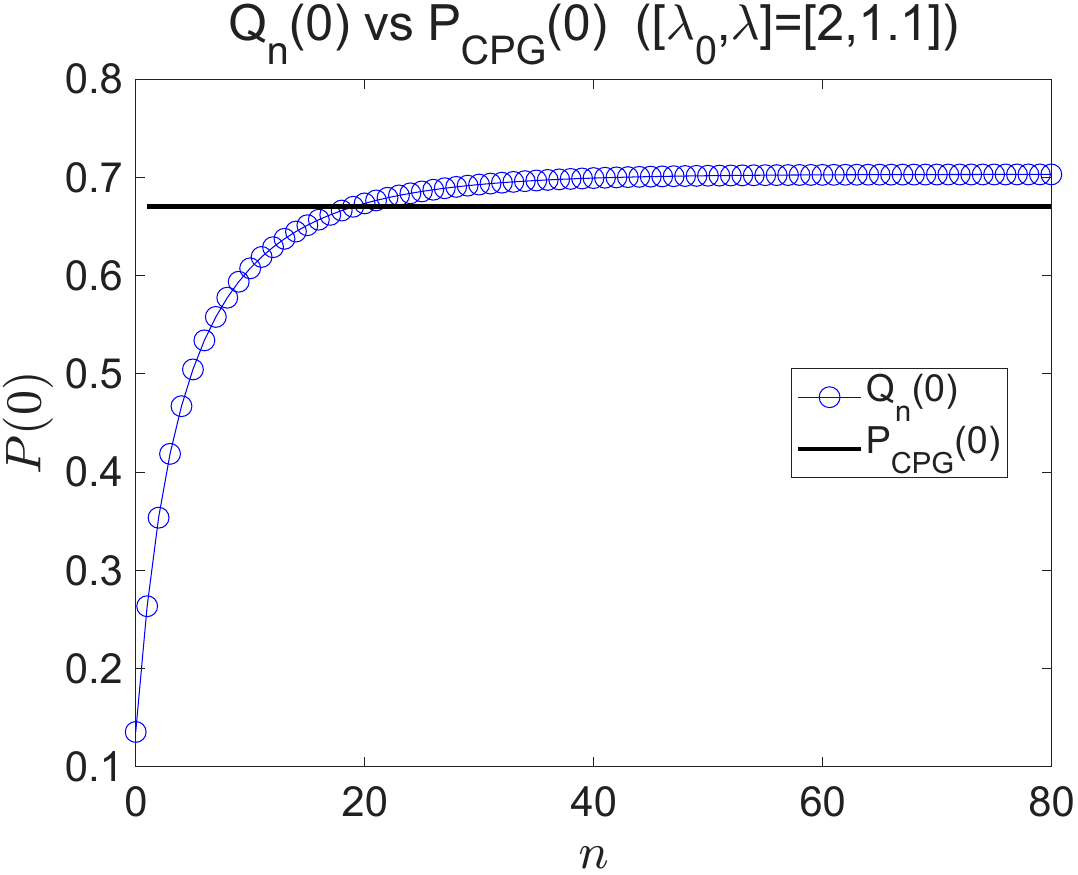}
\includegraphics[width=0.67\columnwidth]{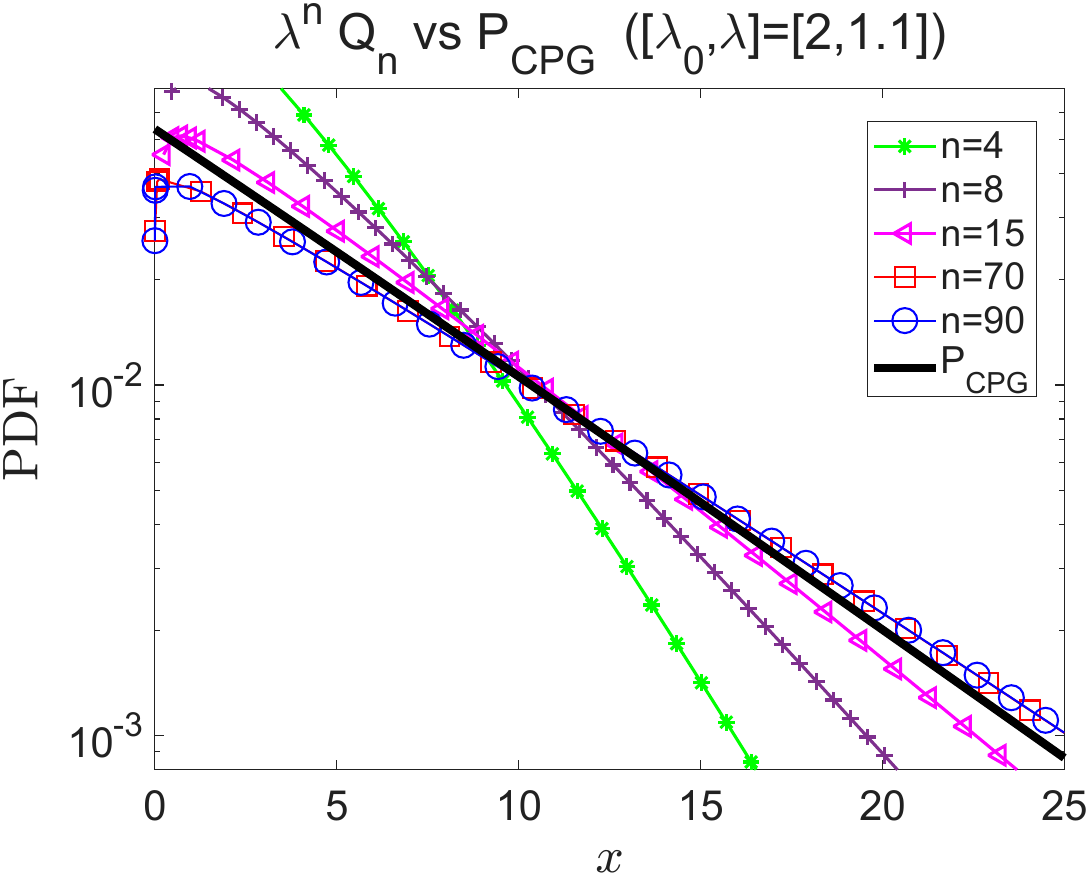}
\includegraphics[width=0.67\columnwidth]{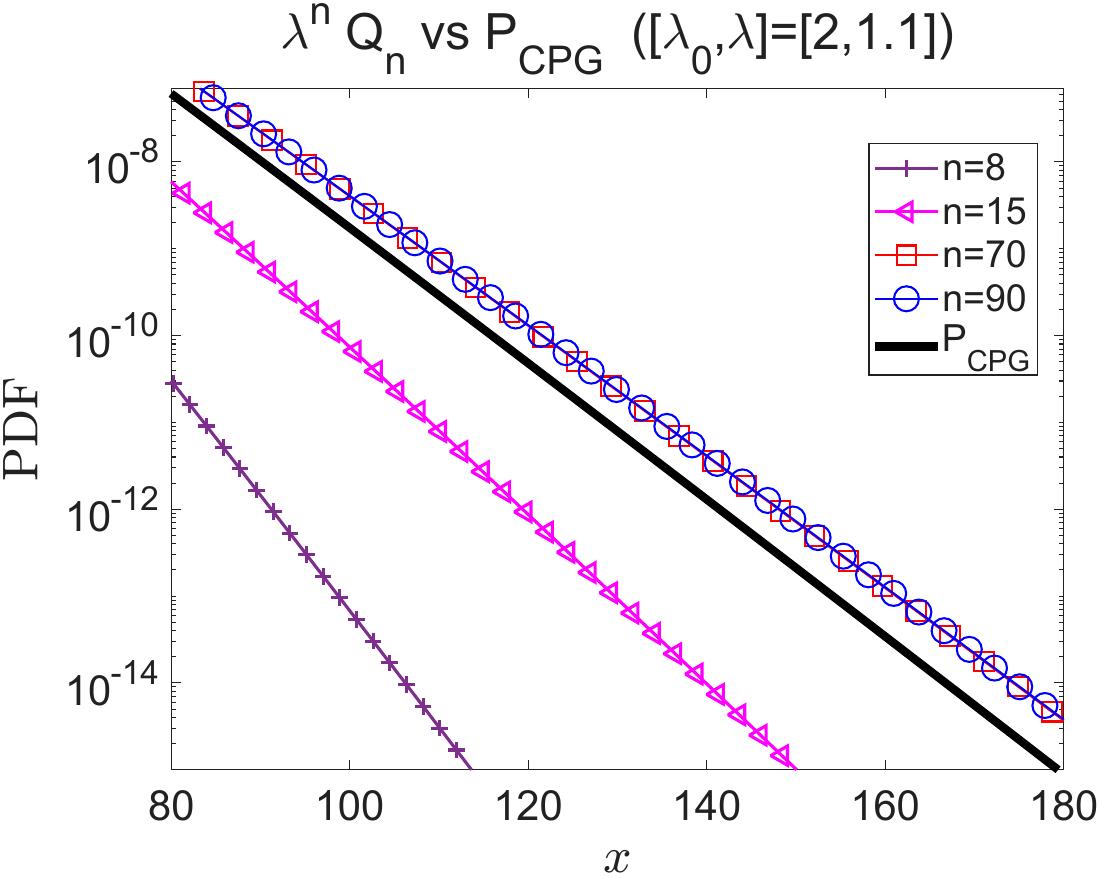}
\includegraphics[width=0.67\columnwidth]{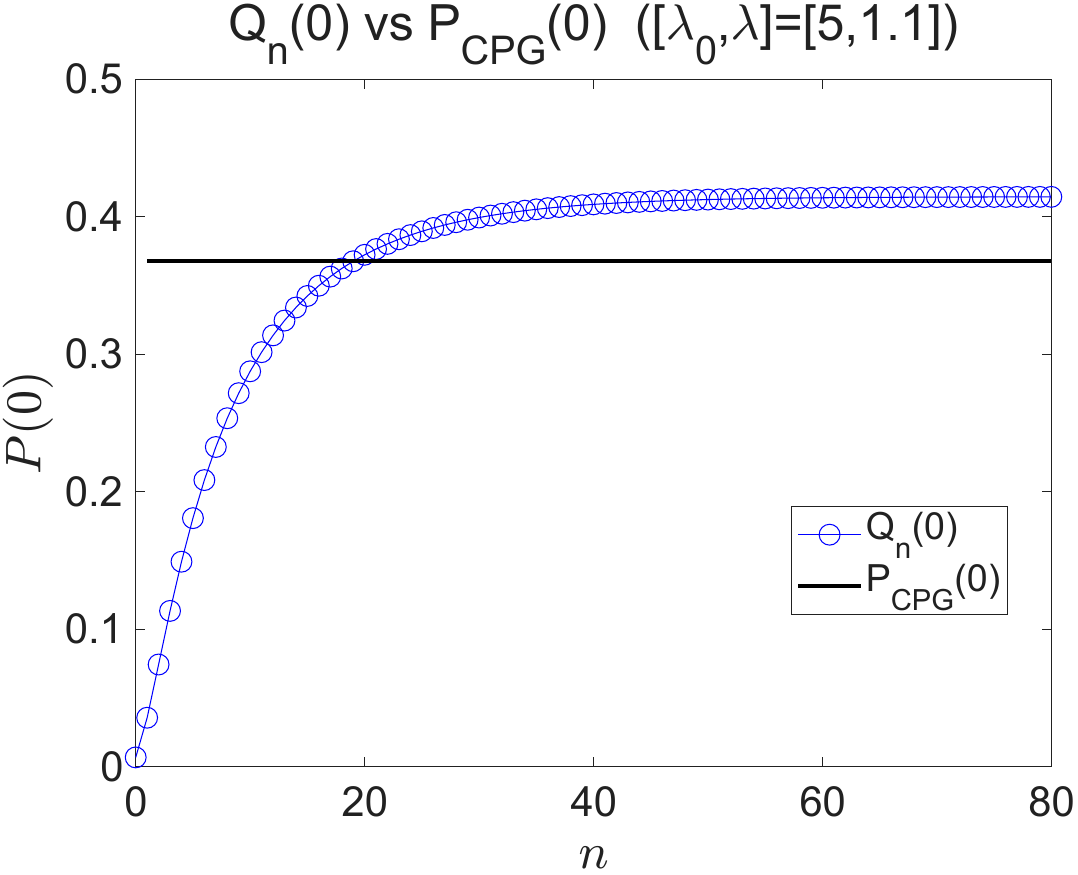}
\includegraphics[width=0.67\columnwidth]{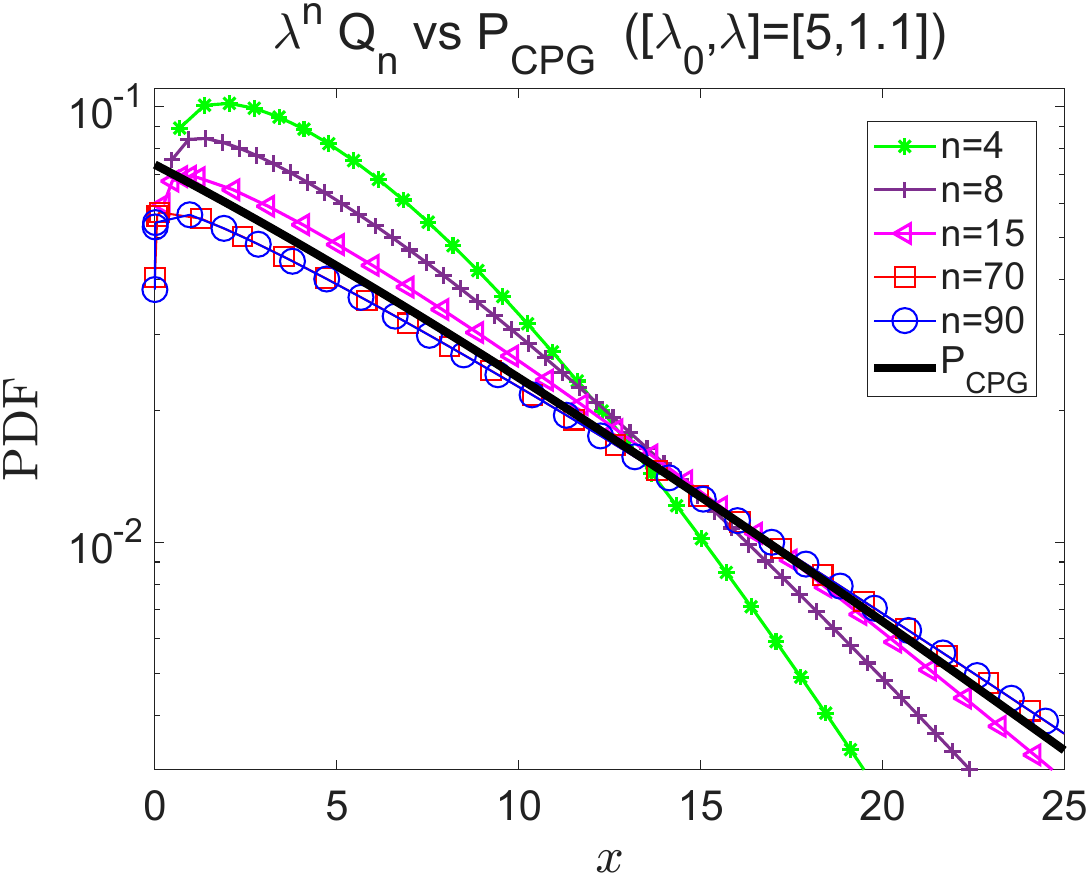}
\includegraphics[width=0.67\columnwidth]{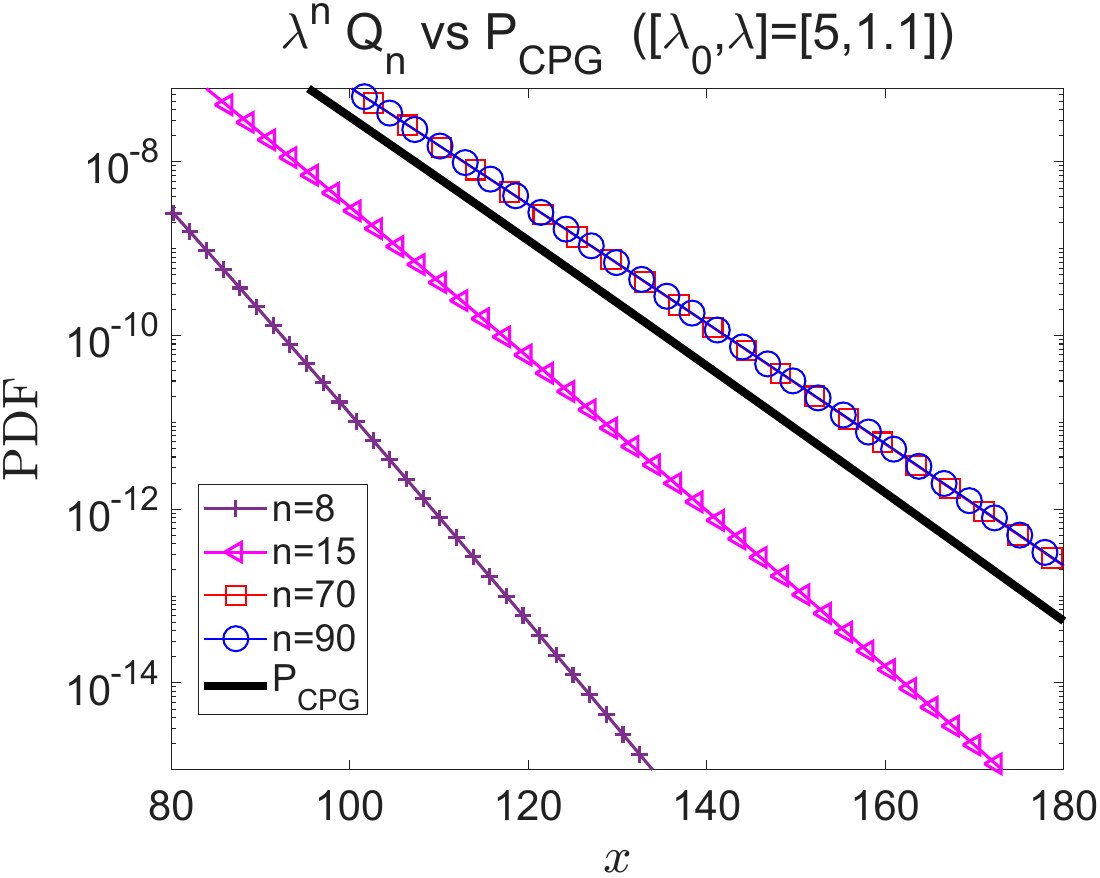}
\caption{Comparison between $P_{\mathrm{CPG}}$ and $Q_n$ in the Poisson case under Condition II at $\lambda=1.1$. Three different values of $\lambda_0$, $1,2,5$, are examined from top to bottom. The left columns are for the plot of $Q_n(0)$ (joined blue circles) against $n$, compared with $P_{\mathrm{CPG}}(0)$ (black line). The other columns compare $\lambda^n Q_n(\lambda^n x)$ and $P_{\mathrm{CPG}}(x)$ in the semi-logarithmic scale. The bulk region and the right tail of the distribution are exhibited in the center and right columns, respectively. } 
\Lfig{condII_close1}
\end{center}
\end{figure*}
As in \Rfigs{Poisson_condI}{Geometric_condI}, both the probability of zero (left panels) and the bulk region of $P_{\rm CPG}(x)$ (center panels) are consistent with $\lambda^n Q_n(\lambda^n x)$ for all values of $\lambda_0$ examined, again supporting our theoretical analysis. 

To further investigate the dependence on $\lambda$, we also examined cases with larger values of $\lambda$. Figure~\ref{fig:condII_far1} shows the plot of $Q_n(0)$ against $n$ and of $\lambda^n Q_n(\lambda^n x)$ in the top and bottom rows, respectively: the value of $\lambda_0$ was fixed to unity and the results with $\lambda=1.5,2,5$ are shown from left to right. 
\begin{figure*}[htbp]
\begin{center}
\includegraphics[width=0.67\columnwidth]{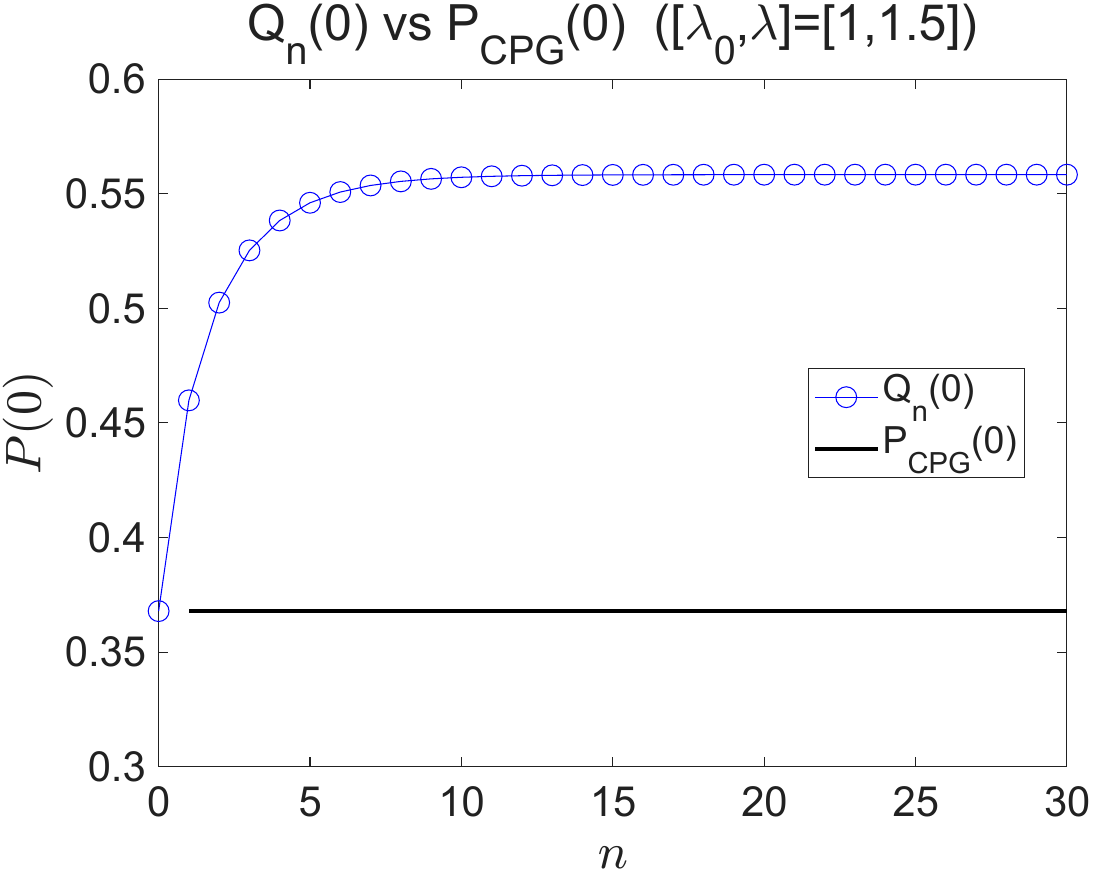}
\includegraphics[width=0.67\columnwidth]{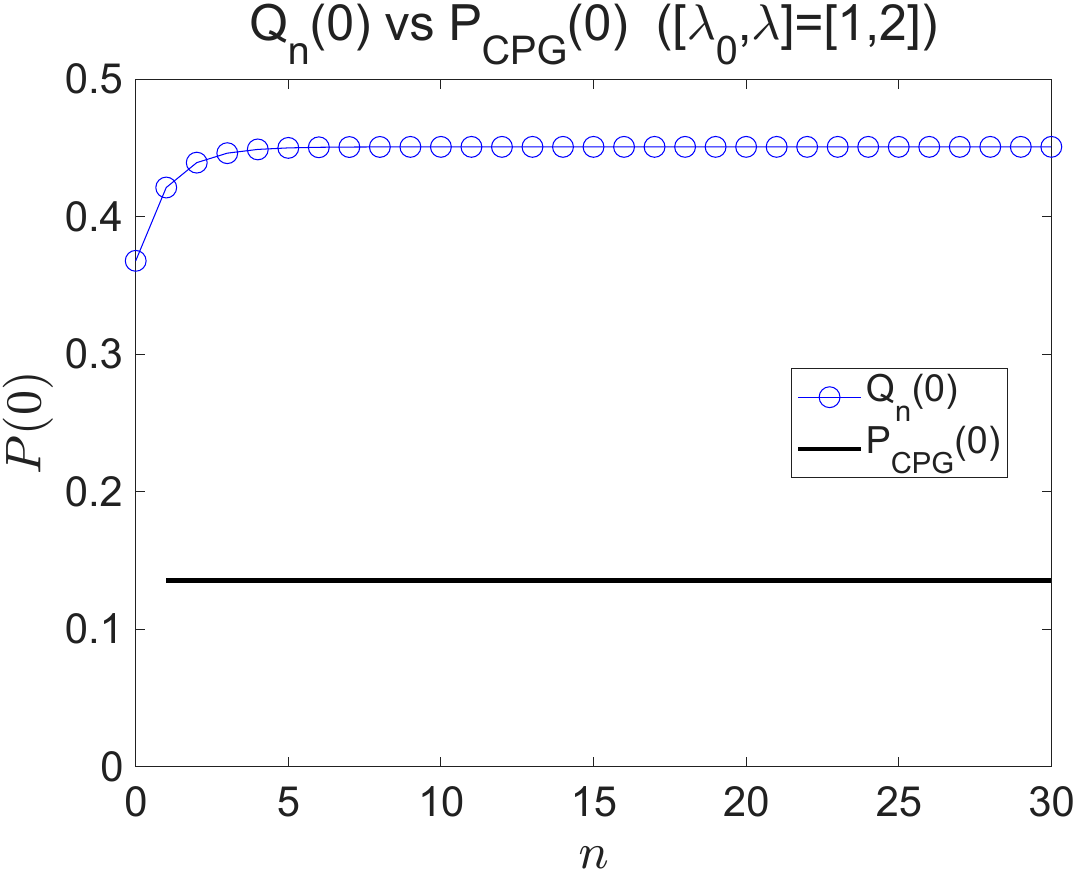}
\includegraphics[width=0.67\columnwidth]{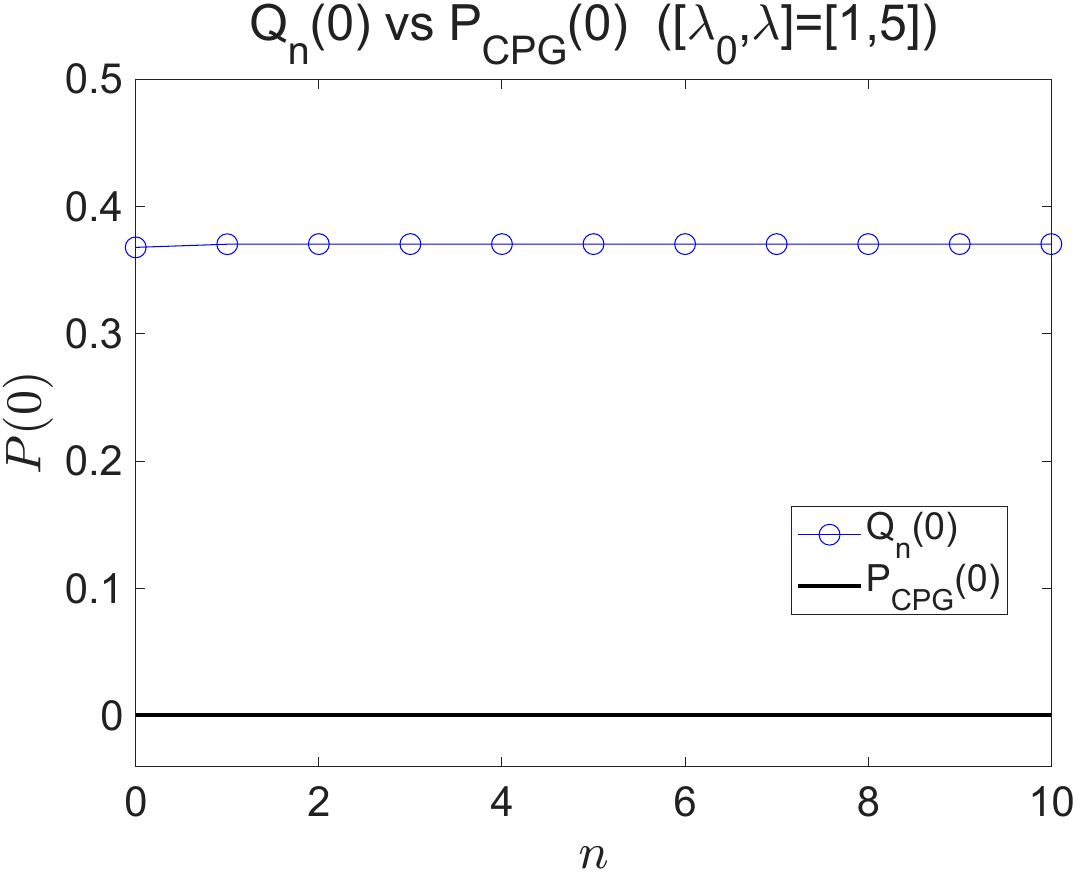}
\includegraphics[width=0.67\columnwidth]{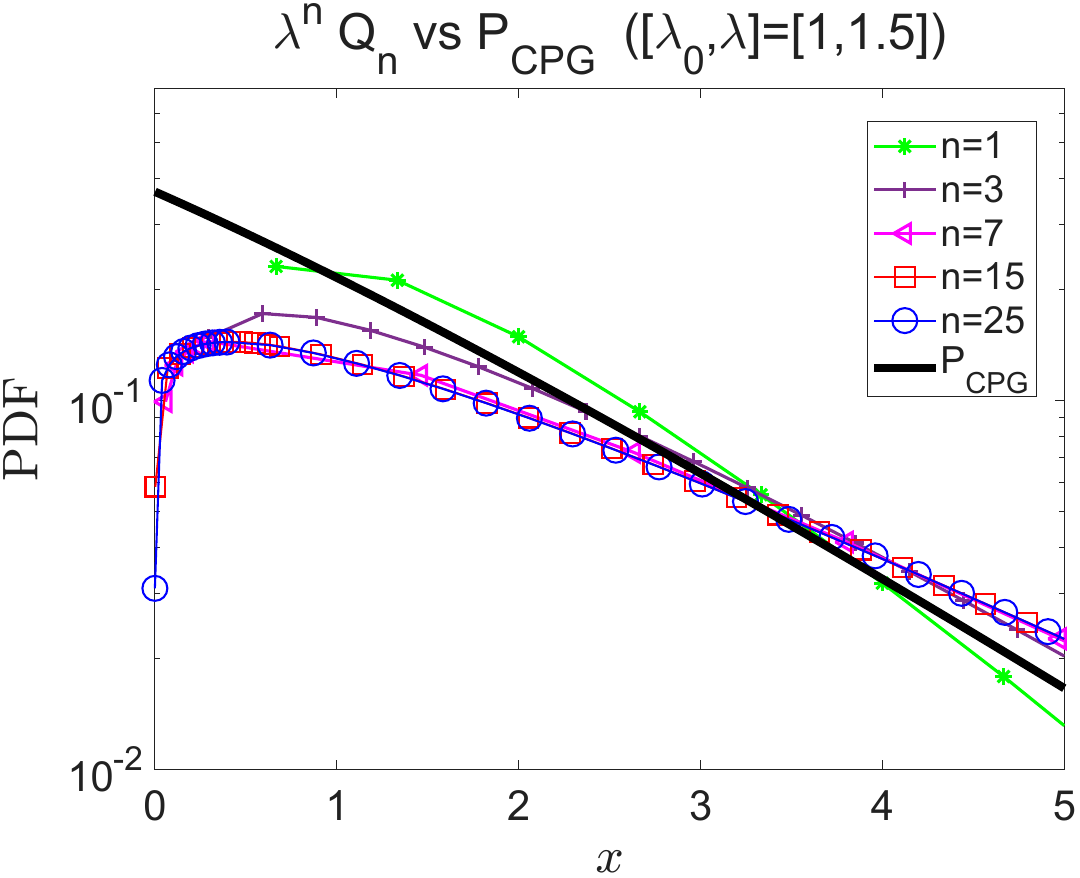}
\includegraphics[width=0.67\columnwidth]{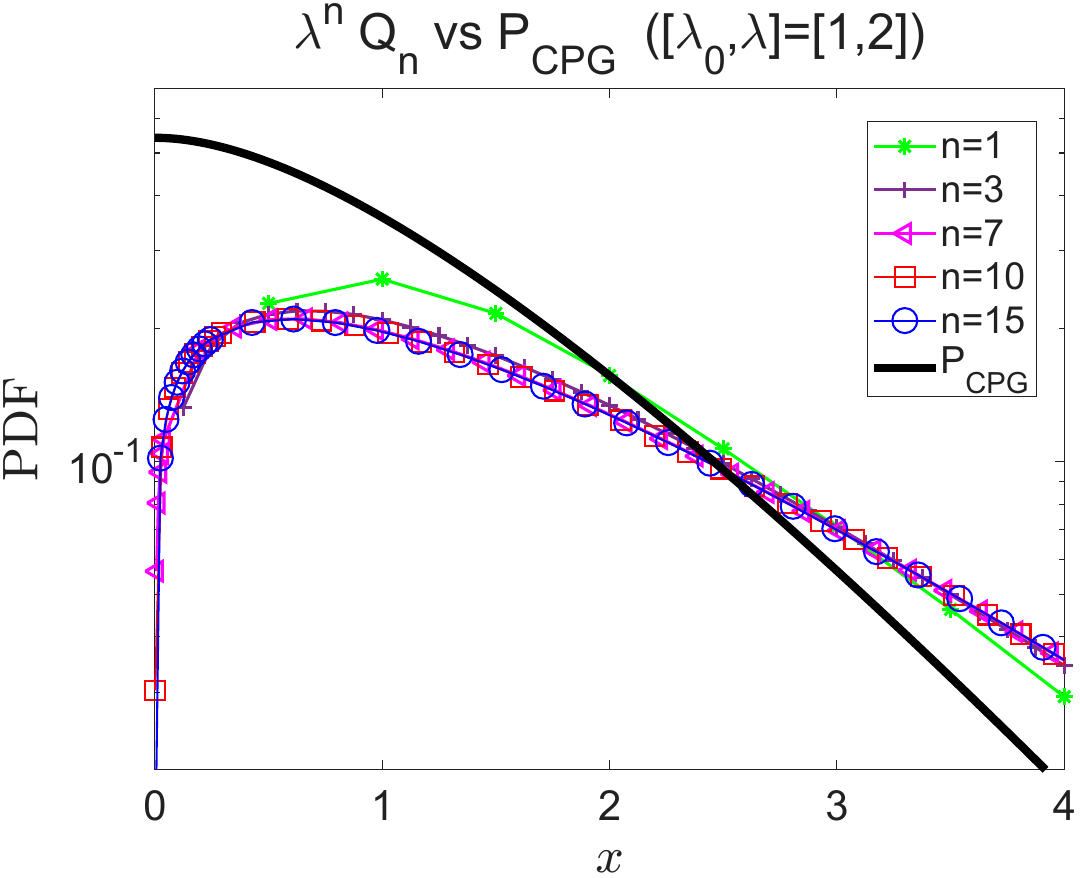}
\includegraphics[width=0.67\columnwidth]{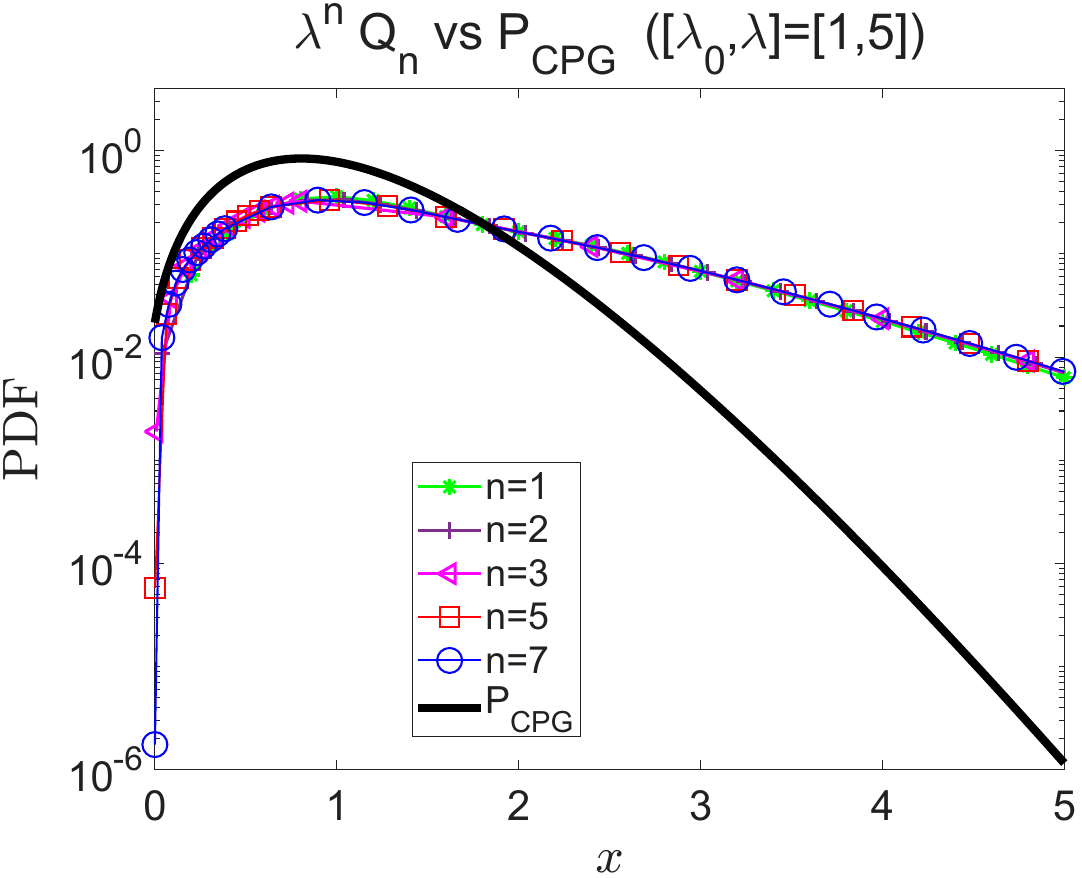}
\caption{Comparison between $P_{\mathrm{CPG}}$ and $Q_n$ in the Poisson case at $\lambda_0=1$. Three different values of $\lambda$, $1.5,2,5$, are examined from left to right. The top row gives the plot of $Q_n(0)$ (joined blue circles) against $n$, compared with $P_{\mathrm{CPG}}(0)$ (black line). The bottom row compare $\lambda^n Q_n(\ell=\lambda^n x)$ (joined color markers) and $P_{\mathrm{CPG}}(x)$ (black line) in the semi-logarithmic scale. } 
\Lfig{condII_far1}
\end{center}
\end{figure*}
As $\lambda$ deviates from $1$, the discrepancy between $Q_n$ and $P_{\mathrm{CPG}}$ increases. An interesting observation in contrast to \Rfig{Poisson_condI} is that the CPG distribution with the theoretical parameter values \Req{CPG param II} is not a good approximation for large $\lambda$, even though the offspring distribution is Poisson. This observation and the one in \Rfig{Geometric_condI} imply that the CPG distribution becomes a good approximation for large $\lambda$ only in rather limited cases.

\subsection{Fitting CPG in the Poisson case} \Lsec{Fitting CPG}
As observed in \Rfigs{Poisson_condI}{condII_far1}, the limiting distribution of the Poisson GW process seems to have a shape similar to the CPG distribution, even when $\lambda$ is not close to 1. This observation suggests that, by appropriately tuning its parameters, the CPG distribution may still work as a good approximation for the limiting distribution even when $\lambda$ significantly deviates from unity. Here, we investigate this possibility.

To this end, we fit the CPG distribution to the numerically computed $Q_n$ at large $n$ and investigate how well the CPG distribution can fit $Q_n$. The fitting procedure we adopted consists of the following two steps:
\begin{enumerate}
\item{Estimate the parameter $\mu$ on the basis of the probability $Q_n(0)$ only, that is, $\hat{\mu} = -\log Q_n(0)$. }
\item{Perform least-squares fitting of $P_{\rm CPG}(\ell / \lambda^n; \hat{\mu}, \alpha, \tau)$ to $\lambda^n Q_n(\ell)$ with respect to $(\alpha,\tau)$ over an appropriate range of $\ell(>0)$, obtaining  $(\hat{\alpha},\hat{\tau})$. }
\end{enumerate}

The fitting results are shown in \Rfig{CPGfit}. 
\begin{figure*}[htbp]
\begin{center}
\includegraphics[width=0.67\columnwidth]{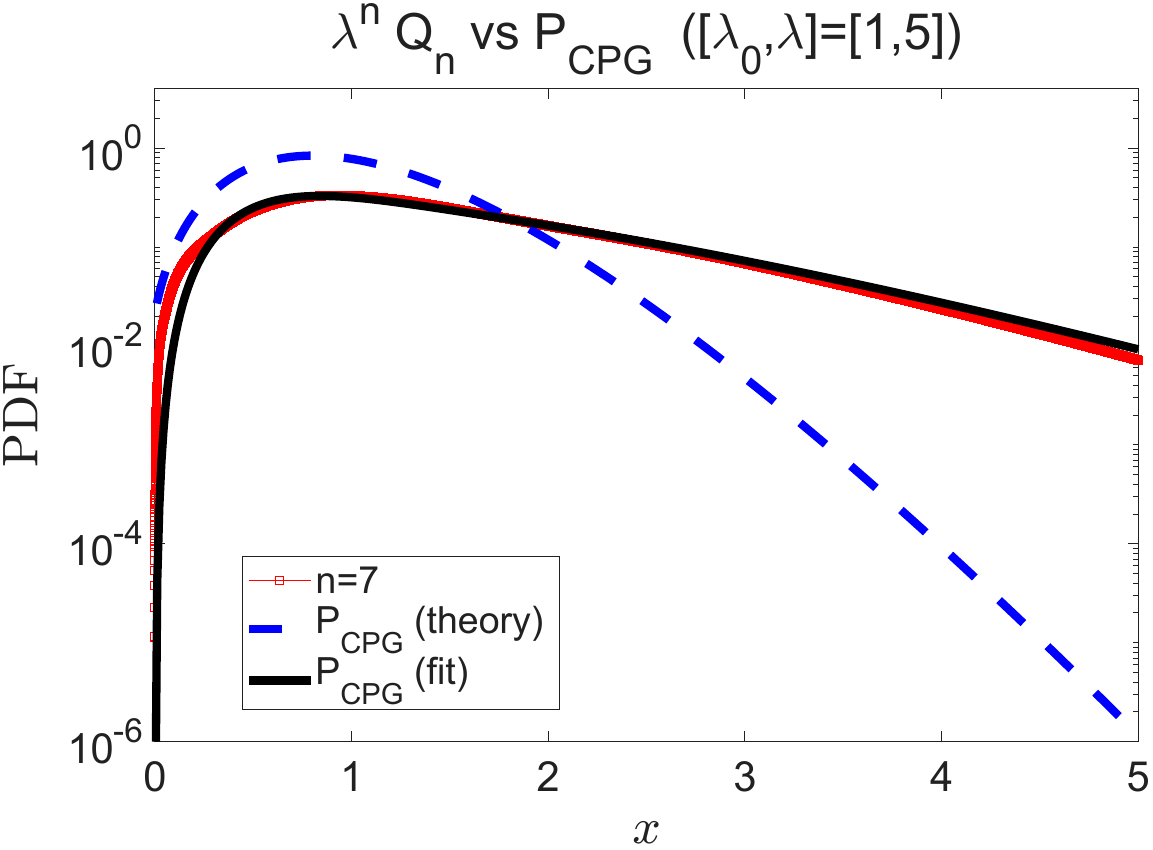}
\includegraphics[width=0.67\columnwidth]{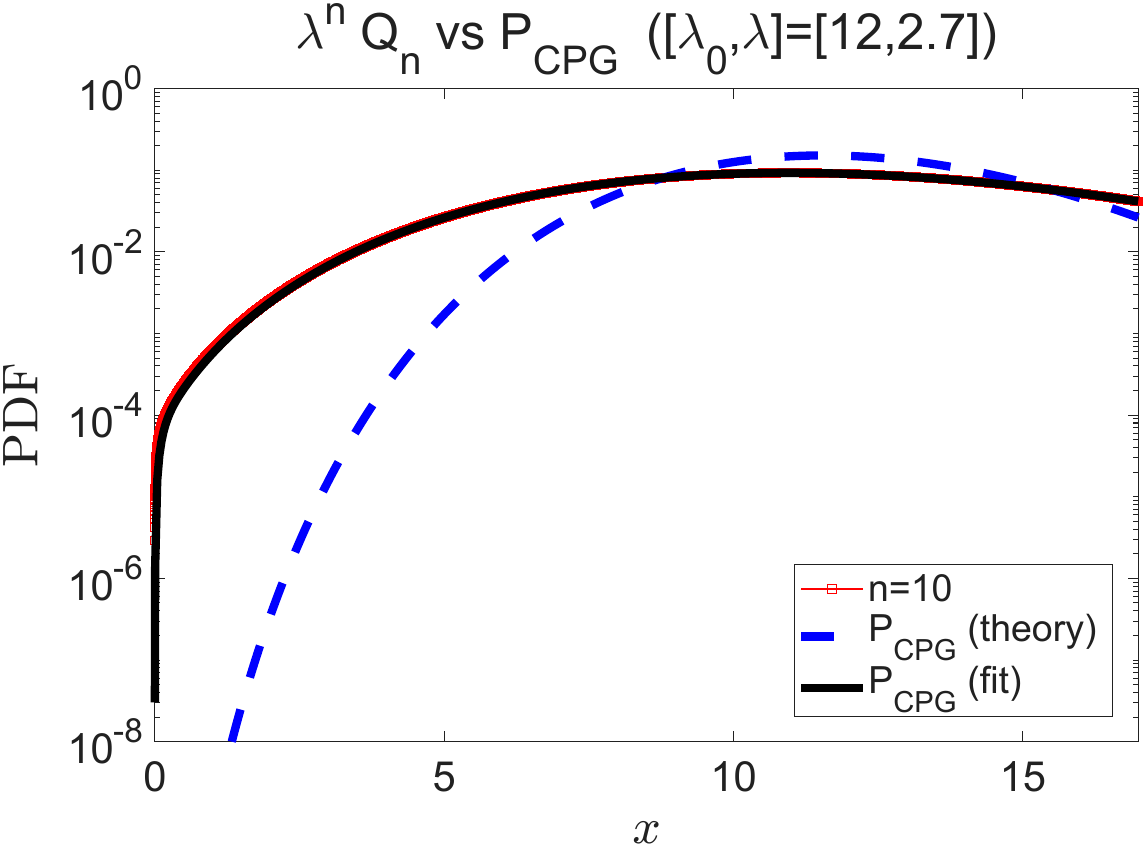}
\includegraphics[width=0.67\columnwidth]{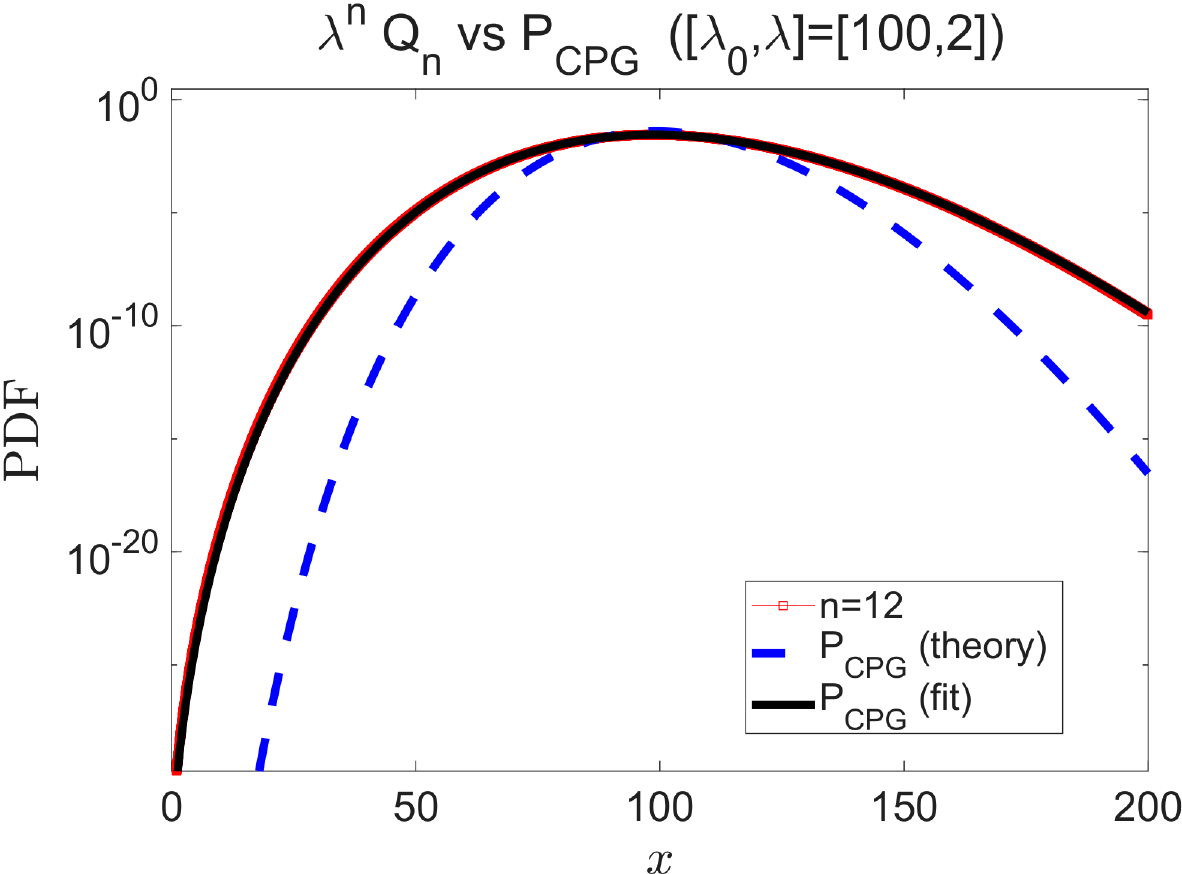}
\caption{Fitted $P_{\mathrm{CPG}}$ (black line), $Q_n$ for large $n$ (joined red squares), and $P_{\mathrm{CPG}}$ with the theoretical parameter values (blue dashed line) are compared in the Poisson case with $\lambda$ fairly larger than $unity$. Three different parameter sets, $(\lambda_0,\lambda)=(1,5)$, $(12,2.7)$, and $(100,2)$ from left to right, are considered. In all the cases, the fitted CPG distribution recovers a nice agreement with $Q_n$. } 
\Lfig{CPGfit}
\end{center}
\end{figure*}
The left panel corresponds to the bottom right panel of \Rfig{condII_far1} in the parameters but compares the fitted CPG (black line) and $Q_n$ for $n = 7$ (red line with squares), as well as the CPG of the parameters with the theoretical values \Req{CPG param II} (blue dashed line). The fitted result showed a nice agreement with $Q_n$ over a wide probability region, though there was a gap in the tail with small probabilities. This suggests that while the CPG distribution is indeed an approximation and does not precisely describe the limiting distribution, it is still useful for approximating regions with high probability values. For further examination, the cases $(\lambda_0,\lambda)=(12,2.7)$ and $(100,2)$ are shown in the middle and right panels, respectively: the former set of the parameter values was taken from the example values of dynode gain of EM in~\cite{Dietz1965}, while the latter one was intended for examining scenarios where $\lambda_0$ is extremely large. In these two cases, the fitted CPGs showed excellent agreement with $Q_n$ at large $n$ even at very small probability values, suggesting that the CPG approximation improves as $\lambda_0$ increases. However, this result may not be entirely surprising because the distributions of both the GW and CPG models approach a normal distribution as $\lambda_0$ and $\hat{\mu}$ increase~\cite[Theorem 3.7.1]{Kaas2008}. 
The theoretical and fitted values of the CPG parameters for these three cases are presented in \Rtab{CPG parameters}.
\begin{table}[h]
  \centering
\caption{The theoretical and fitted values of CPG parameters. }
\begin{tabular}{ccc}
\hline \hline
 &
  \begin{tabular}[c]{@{}c@{}}Theory\\ $(\hat{\mu},\hat{\alpha},\hat{\tau})$\end{tabular} &
  \begin{tabular}[c]{@{}c@{}}Fitted\\ $(\hat{\mu},\hat{\alpha},\hat{\tau})$\end{tabular} \\ \hline
$(\lambda_0,\lambda)=(1,\,5)$  & (8,\,1,\,0.125)    & (0.99,\,3.80,\,0.275) \\
$(\lambda_0,\lambda)=(12,\,2.7)$  & (40.8,\,1,\,0.294) & (11.0,\,2.19,\,0.500) \\
$(\lambda_0,\lambda)=(100,\,2)$  & (200,\,1,\,0.5)    & (79.7,\,1.68,\,0.745) \\ \hline \hline
\end{tabular}
\Ltab{CPG parameters}
\end{table}

Overall, the observation in this section supports the use of the CPG distribution to describe the limiting distribution of the Poisson GW process even when $\lambda$ is not close to unity.
This result suggests the potential of the CPG distribution to model data distributions generated from supercritical GW processes with any $\lambda$ in practical applications.

\section{Discussion}\Lsec{Discussion}

\subsection{Tail behaviors}
\label{sec:tail}
In general, if the moment generating function $M_X(t)=\mathbb{E}(e^{tX})$,
or the cumulant generating function $\log M_X(t)$,
of a random variable $X$ diverges at $t=t_0>0$,
it implies that the tail probability $\mathbb{P}(X>x)$ of $X$ behaves
as $\mathbb{P}(X>x)\sim Ce^{-t_0x}$ with some $C>0$ as $x\to\infty$.
In other words, the distribution of $X$ has an exponential tail. 
A precise statement is given in the next proposition. 
\begin{prop}
  \label{prop:tail}
  Let $X$ be a nonnegative random variable,
  $F_X(x):=\mathbb{P}(X\le x)$ be the cumulative distribution function of $X$,
  and $M_X(t)=\mathbb{E}(e^{tX})$ be the moment generating function of $X$. 
  Let $\bar{F}_X(x):=1-F_X(x)$.
  Then:
  \begin{enumerate}
  \item[(a)] Assume $M_X(t)<\infty$ for some $t>0$.
    One then has $\lim_{x\to\infty}e^{tx}\bar{F}_X(x)=0$.
  \item[(b)] Assume $\sup_{x}e^{tx}\bar{F}_X(x)<\infty$
    for some $t>0$.
    Then for any $t'\in[0,t)$ one has $M_X(t')<\infty$.
  \end{enumerate}
\end{prop}
A proof of Proposition~\ref{prop:tail} is given in
Appendix~\ref{sec:prf_tail}.

Tail behavior of the distribution of the limiting random variable $W$ is
linked to that of the distribution of the offspring distribution,
as shown in the following proposition. 
\begin{prop}[Proposition 3 in~\protect\cite{Biggins1981}]
  \label{prop:Biggins}
  Let $f(s)$ and $M_{W}(t)$ be the probability generating function
  of the offspring distribution and the moment generating function
  of $W$. Let 
  \begin{align}
    s_0&:=\sup\{s\ge0:f(s)<\infty\},
    \\
    t_0&:=\sup\{t:M_{W}(t)<\infty\}.
  \end{align}
  One then has:
  \begin{itemize}
  \item[(a)] $s_0=\infty$ if and only if $t_0=\infty$.
  \item[(b)] $s_0\in(1,\infty)$ if and only if $t_0\in(0,\infty)$.
  \end{itemize}
\end{prop}
Propositions~\ref{prop:tail} and \ref{prop:Biggins} together
imply that the distribution of the limiting random variable $W$
has an exponential tail if the offspring distribution
has an exponential tail,
as well as that the former has a lighter-than-exponential tail
if the latter has a lighter-than-exponential tail.

Returning to our asymptotic results, one observes that the moment generating
functions in both Conditions I and II diverge at
$t=2(\lambda-1)/\kappa_2^*<\infty$, 
implying that the corresponding distribution has an exponential tail
regardless of the offspring distribution.
This is in conflict with Proposition~\ref{prop:Biggins},
showing that our asymptotic result does not in general reproduce the
true tail behavior of $W$.

\subsection{Diffusion approximation}
Diffusion approximation of the GW processes~\cite{Feller1951,Jirina1969,Lindvall1972,Lindvall1974} is closely related with our framework.
The diffusion approximation of the GW process is obtained via
rescaling of the state space and time.
More concretely, one counts generation $n$
as well as the population size $Z_n$ as multiples of
$\Delta\gg1$, via
\begin{equation}
  Y_s=\frac{Z_{s\Delta}}{\Delta},
\end{equation}
and studies the diffusion process defined in the limit $\Delta\to\infty$,
which is called the diffusion branching process. 

Assume a family of the diffusion branching processes
indexed by $\Delta>0$.
For the offspring distribution,
we assume $\kappa_1=\lambda=1+\alpha/\Delta$ with $\alpha\in(-\infty,\infty)$,
and $\kappa_2^*:=\lim_{\Delta\to\infty}\kappa_2\in(0,\infty)$. 
We also assume $Y_0=z_0\Delta$ with $z_0\in[0,\infty)$. 
The Feller-Jir\v{\i}na theorem~\cite{Feller1951,Jirina1969}
(see also \cite{Lindvall1974}) states that,
under some regularity conditions,
the cumulant generating function $\xi(t)$ of $Y_s$
in the limit $\Delta\to\infty$ is given by 
\begin{equation}
  \label{eq:FellerJirina}
  \xi(t)=\frac{z_0te^{\alpha s}}{1-\frac{\kappa_2^*u_s(\alpha)}{2}t},
\end{equation}
where 
\begin{equation}
  u_s(\alpha)=\begin{cases}
  \frac{e^{\alpha s}-1}{\alpha},&\alpha\not=0,\\
  s,&\alpha=0.
  \end{cases}
\end{equation}
One can observe that the state distribution
of the diffusion branching process in the criticality limit
is again given by a CPG distribution. 

It should be noted that in deriving the above result
one takes the limit $\Delta\to\infty$ while
$(\lambda-1)n=\frac{\alpha}{\Delta}\cdot s\Delta=\alpha s$
is kept finite.
In other words, one takes the limits $\lambda\to1$
and $n\to\infty$ simultaneously.
In our argument, on the other hand,
we take the limit $n\to\infty$ first,
and then consider the asymptotic $\lambda\downarrow1$.
It is interesting to observe that, in spite of
the difference in how to take the limits,
one arrives at the same CPG distribution.
It should also be noted that the above result states that 
the distribution of $Y_s$ in the limit $\Delta\to\infty$
has an exponential tail, just as our asymptotic result. 
As argued in Section~\ref{sec:tail},
it may not be consistent with the limiting distribution
of $W$, which has a lighter-than-exponential tail
if the offspring distribution has a lighter-than-exponential tail.

\subsection{Large-$n$ asymptotic}
\label{sec:largen}
\cite{GarciaMillanFontClosCorral2015} performed an analysis on the finite-size scaling of the survival probability
$\mathbb{P}(Z_n>0)$ of the GW branching process by considering sufficiently large but finite $n = m + \ell$ with $\ell \gg m \gg 1$.
One can extend their analysis to discuss the large-$n$ asymptotic
of the cumulant generating function $K_{\bar{W}_{n}}(t)$ of
$\bar{W}_n=(\lambda-1)Z_n/\lambda^n$ as $\lambda\downarrow1$.

For the analysis below, we impose the following assumptions on the offspring distribution.
\begin{asm}
\label{asm:offspring}
\leavevmode
\begin{enumerate}[(i)]
\item The offspring distribution is parameterized
  by its mean $\lambda$. 
\item The cumulant generating function
  \begin{equation}
    \psi(t)=\sum_{\nu=1}^\infty\frac{\kappa_\nu}{\nu!}t^\nu
    \Leq{CGF_off}
  \end{equation}
  of the offspring distribution exists, where $\kappa_\nu$ denotes the $\nu$th cumulant of the offspring distribution and $\kappa_1 = \lambda$.
  \label{asm:i}
\item 
$\lim_{\lambda \downarrow 1}\kappa_2 = \kappa_2^* \in (0,\infty)$. 
\label{asm:ii}
\item
For any $\nu \geq 3$, the order of the $\nu$th cumulant $\kappa_\nu$ is $o(1/(\lambda-1)^{\nu-2})$ in the limit $\lambda \downarrow 1$.
\label{asm:iii}
  \end{enumerate}
\end{asm}
Several common discrete probability distributions, including the Poisson and negative binomial distributions, satisfy these assumptions. 
It should also be noted that these assumptions are weaker than those in \Rsec{Condition I}.

Under Assumption~\ref{asm:offspring} and Condition I, $K_{\bar{W}_{m}}(t)$ satisfies the functional equation
\begin{equation}
  \label{eq:recK}
  K_{\bar{W}_{m+1}}(\lambda t)=\psi(K_{\bar{W}_{m}}(t)),
\end{equation}
where $\psi(t)$ denotes the cumulant generating function
of the offspring distribution.
Near the criticality, we may assume that, for sufficiently large $m$, $K_{\bar{W}_{m}}(t)$
is close to $(\lambda-1)t/(1-\frac{\kappa_2^*}{2}t)$,
which is $\mathcal{O}(\lambda-1)$ as $\lambda \downarrow 1$.
Under this assumption, one may take account of
terms up to second order in $(\lambda-1)$ while ignoring
higher-order terms in~\eqref{eq:recK}, yielding 
\begin{align}
  \label{eq:rec}
  \frac{1}{K_{\bar{W}_{m+1}}(\lambda t)}
  &=\frac{1}{\psi(K_{\bar{W}_{m}}(t))}
  \nonumber\\
  &\approx \frac{1}{\lambda K_{\bar{W}_{m}}(t)+\frac{\kappa_2^*}{2}K_{\bar{W}_{m}}(t)^2}
  \nonumber\\
  &=\frac{1}{\lambda K_{\bar{W}_{m}}(t)\left(1+\frac{\kappa_2^*K_{\bar{W}_{m}}(t)}{2\lambda}\right)}
  \nonumber\\
  &\approx \frac{1}{\lambda K_{\bar{W}_{m}}(t)}
  \left(1-\frac{\kappa_2^*K_{\bar{W}_{m}}(t)}{2\lambda}\right)
  \nonumber\\
  &=\frac{1}{\lambda K_{\bar{W}_{m}}(t)}
  -\frac{\kappa_2^*}{2\lambda^2}.
\end{align}
Note that, in both the second and fourth lines, the approximations are performed by noting that $K_{\bar{W}_m}(t)=\mathcal{O}(\lambda-1)$ is a small quantity as $\lambda\downarrow1$, and by retaining terms up to the second order in $(\lambda-1)$.

Since
\begin{equation}
  \frac{1}{K_{\bar{W}_{m}}(t)}=\frac{1-\frac{\kappa_2^*}{2\lambda}t}{(\lambda-1)t}
\end{equation}
is the fixed point of~\eqref{eq:rec}, 
we let
\begin{equation}
  h_m(t):=\frac{1}{K_{\bar{W}_{m}}(t)}-\frac{1-\frac{\kappa_2^*}{2\lambda}t}{(\lambda-1)t}
  \label{eq:def_h_n}
\end{equation}
be the deviation of $1/K_{\bar{W}_{m}}(t)$ from the fixed point. 
The above equation can then be rewritten in terms of $h_m$ as
the following simple formula: 
\begin{equation}
  h_{m+1}(\lambda t)=\frac{1}{\lambda}h_m(t).
\end{equation}
Let $\mathcal{F}$ be the operator that maps $h(t)$ to $\frac{1}{\lambda}h(\frac{t}{\lambda})$.
The above equation can then be represented concisely as
$h_{m+1}=\mathcal{F}(h_m)$.
It should be noted that $\mathcal{F}$ is a linear operator,
which can be confirmed via
\begin{align}
  \mathcal{F}((af+bg)(t))
  &=\frac{1}{\lambda}(af+bg)\left(\frac{t}{\lambda}\right)
  \nonumber\\
  &=\frac{a}{\lambda}f\left(\frac{t}{\lambda}\right)+\frac{b}{\lambda}g\left(\frac{t}{\lambda}\right)
  \nonumber\\
  &=a\mathcal{F}(f(t))+b\mathcal{F}(g(t)).
\end{align}
One also has
\begin{equation}
  \mathcal{F}(t^k)=\frac{1}{\lambda}\left(\frac{t}{\lambda}\right)^k
  =\frac{t^k}{\lambda^{k+1}},
\end{equation}
showing that $t^k$ is the eigenfunction of $\mathcal{F}$
associated with the eigenvalue $1/\lambda^{k+1}$.

It is known that $\bar{W}_m$ has mean $\lambda-1>0$, 
which allows us to apply Wronski's formula~\cite[Theorem 1.3]{Henrici1974}
to a power series expansion
of $K_{\bar{W}_{m}}(t)/t$.
One consequently finds that 
$h_m(t)$ admits the following series expansion.
\begin{equation}
  h_m(t)=\sum_{k=0}^\infty c_k t^{k-1}.
  \label{eq:power_series_h_n}
\end{equation}
Repeated application of the operator $\mathcal{F}$ to $h_m$ yields
\begin{equation}
  h_{m+\ell}(t)=\mathcal{F}^\ell(h_m)
  =\sum_{k=0}^\infty\frac{c_k}{\lambda^{\ell k}}t^{k-1}.
\end{equation}
This in turn can be translated into the following formula
for $K_{\bar{W}_{m+\ell}}(t)$:
\begin{align}
  &K_{\bar{W}_{m+\ell}}(t)
  =\frac{1}{h_{m+\ell}(t)+\frac{1-\frac{\kappa_2^*}{2\lambda}t}{(\lambda-1)t}}
  \nonumber\\
  &=\frac{(\lambda-1)t}{1+c_0(\lambda-1)
    -\left(\frac{\kappa_2^*}{2\lambda}-\frac{c_1(\lambda-1)}{\lambda^\ell}\right)t
  +\sum_{k=2}^\infty\frac{c_k(\lambda-1)}{\lambda^{\ell k}}t^k}.
\label{eq:bar_K_mplusl_1}
\end{align}
The coefficients $c_0$ and $c_1$ are given by
\begin{align}
c_0 &= 0, \label{eq:c_0} \\
c_1 &= \frac{\kappa_2^*}{2\lambda(\lambda-1)} - \frac{\kappa_2 \sum_{i=1}^{m} \lambda^{i-1}}{2 \lambda^{m+1}} , \label{eq:c_1}
\end{align}
respectively (see Appendix~\ref{sec:order_c_k} for the derivation).
Therefore, \eqref{eq:bar_K_mplusl_1} can be rewritten as
\begin{align}
  &K_{\bar{W}_{m+\ell}}(t) \nonumber \\
  &=\frac{(\lambda-1)t}{1
    - \{ \frac{\kappa_2^* ( 1 - \lambda^{-\ell})}{2\lambda} + \frac{\kappa_2 \sum_{i=1}^{m} \lambda^{i-1}}{2 \lambda^{m+\ell+1}} (\lambda-1) \} t
  +\sum_{k=2}^\infty\frac{c_k(\lambda-1)}{\lambda^{\ell k}}t^k}.
\end{align}
We now fix $m$ and assume the asymptotic of $\ell$ tending large,
via the scaling defined by $\ell=s\Delta$ and $\lambda-1=\alpha/\Delta$
in the limit $\Delta\to\infty$.
This scaling yields $\lambda=1+\alpha/\Delta$
and $\lambda^\ell\approx e^{\alpha s}$. 
Moreover, the order of $c_k$ for $k \geq 2$ is $o(1)$ in the asymptotic regime $\lambda \downarrow 1$ (see Appendix~\ref{sec:order_c_k} for a detailed discussion).
One therefore obtains 
\begin{equation}
   K_{\bar{W}_{n}}(t) = K_{\bar{W}_{m+\ell}}(t) 
  \approx \frac{\alpha t/\Delta}{1-\frac{\kappa_2^*}{2}\left(1-e^{-\alpha s}\right)t}.
\end{equation}
This result suggests that when $\lambda$ is larger than, but close enough to,
1, the distribution of $\bar{W}_n$ for large $n=s\Delta$ will be well
approximated by the CPG distribution whose cumulant generating function
is given by
\begin{equation}
  \tilde{K}(t)=\frac{\alpha t/\Delta}{1-\frac{\kappa_2^*}{2}(1-e^{-\alpha s})t}.
\end{equation}
The cumulant generating function $\xi(t)$ 
of $Y_s:=Z_{s\Delta}/\Delta$ with the initial condition
$Y_0=z_0\Delta$ and in the limit $\Delta\to\infty$ is derived
from the above result via noting that
\begin{align}
  Y_s&=\frac{Z_{s\Delta}}{\Delta}
  =\frac{1}{\Delta}\frac{\lambda^{s\Delta}}{\lambda-1}\bar{W}_{s\Delta}
  =\frac{1}{\Delta}\frac{(1+\alpha/\Delta)^{s\Delta}}{\alpha/\Delta}\bar{W}_{s\Delta}
  \nonumber\\
  &\stackrel{\Delta\to\infty}{\xrightarrow{\hspace{10mm}}}
  \frac{e^{\alpha s}}{\alpha}\bar{W}_{s\Delta},
\end{align}
as 
\begin{equation}
  \xi(t)=\frac{z_0te^{\alpha s}}{1-\frac{\kappa_2^*(e^{\alpha s}-1)}{2\alpha}t}.
\end{equation}
It should be noted that the above result is identical to
the formula~\eqref{eq:FellerJirina} in the Feller-Jir\v{\i}na theorem
for the diffusion branching processes.

\begin{table*}
  \centering
  \caption{Asymptotic and limiting behaviors of GW branching processes
    near criticality. Our results are colored in blue.
    ``$X\sim\mbox{Dist}$'' means that random variable $X$ follows
    the distribution Dist.
    ``$X\mid Y>0\sim\mbox{Dist}$'' means that random variable $X$ follows
    the distribution Dist.\ conditionally on the event $Y>0$.}
  \label{tab:asym}
  \begin{tabular}{|c|c|c|}
    \hline 
    $\lambda$ & $n$ large but finite & $n=\infty$
    \\\hline
    \strut Supercritical & \multirow{2}{*}{$\color{blue}\frac{Z_n}{\lambda^n}\sim\mbox{CPG}$~(Section~\ref{sec:largen})}  & Explosion: $\lim_{n\to\infty}Z_n\mid Z_n>0=\infty$~\cite[Theorem 6.2]{Harris1963}
    \\
    $\lambda\downarrow1$ & & $\color{blue}\lim_{n\to\infty}\frac{Z_n}{\lambda^n}\sim\mbox{CPG}$~(Section~\ref{sec:Theoretical Results})
    \\\hline
    \strut Critical & \multirow{4}{*}{$Z_n\mid Z_n>0\sim\mbox{Geom.}$~\cite[Section 3.4]{Sevastyanov1971}} & Extinction: $\lim_{n\to\infty}Z_n=0$~\cite[Theorem 6.1]{Harris1963}
    \\
    $\lambda=1$ & & $\lim_{n\to\infty}\frac{Z_n}{n}\mid Z_n>0\sim\mbox{Expn.}$~\cite[Theorem 10.1]{Harris1963}
    \\\cline{1-1}\cline{3-3}
    \strut Subcritical & & Extinction: $\lim_{n\to\infty}Z_n=0$~\cite[Theorem 6.1]{Harris1963}
    \\
    $\lambda\uparrow1$ & & $\lim_{n\to\infty}Z_n\mid Z_n>0\sim\mbox{Geom.}$~\cite[Section 3.4]{Sevastyanov1971}
    \\\hline
  \end{tabular}
\end{table*}

Table~\ref{tab:asym} summarizes known results, as well as our contributions,
about asymptotic and limiting behaviors
of the GW process in the criticality limit $\lambda\to1$. 

\section{Conclusion}\Lsec{Conclusion}
In this work, we studied the asymptotics of supercritical GW processes as they approach criticality, $\lambda \downarrow 1$.
We derived a probability distribution in closed form whose cumulant generating function matches that of the limiting distribution of the GW process up to the leading-order in $\lambda-1$ and showed that the derived distribution belongs to a class of CPG distributions.
Numerical experiments revealed that these CPG models were in good agreement with the corresponding GW models for sufficiently large generations under a reasonable parameter regime.
We also showed that, even when $\lambda$ is not close to unity and the offspring distribution is Poisson, the appropriately fitted CPG model closely matched the corresponding GW model.

In the discussion section, we argued the tail behaviors of the distribution of the limiting random variable of the supercritical GW process.
While this discussion implies the limitation of the approximation capability of the CPG model, it should be noted that this limitation may not pose a significant problem in practical applications: 
for example, one is typically interested more in bulk behaviors
rather than in tail behaviors in analysis of data from EMs. 
We furthermore provided discussion to clarify the position and contribution of our work in comparison to other related studies.
They together suggest that the CPG distribution would provide
a good approximation of the population-size distribution of
the GW processes in the asymptotic $\lambda\downarrow1$. 

Our results suggest the potential of the CPG model to analyze datasets generated from cascaded multiplication processes, such as detection signals of EMs and population sizes of individuals with specific biological characteristics, offering various mathematical and statistical advantages for data analysis.
In particular, its application holds promise for large-scale datasets, which have been challenging to analyze using supercritical GW models.

\section*{Acknowledgments}\Lsec{Acknowledgements}
This work was partially supported by JST, CREST Grant number JPMJCF1862, Japan (KU, TO, TT), JSPS KAKENHI under Grant 22K12179 (TO), Grant 21J23154
(KU),
Grant-in-Aid
for Transformative Research Areas (A), ``Foundation of machine learning physics'' (22H05117) (TO), and 
Grant-in-Aid for Transformative Research Areas (A), ``Shin-biology regulated by protein lifetime'' (24H01895) (TT).

\section*{DATA AVAILABILITY}\Lsec{Data_availability}
The codes to reproduce the experimental results reported in this article are openly available~\cite{Uemura2025zenodo}.


\appendix
\section{Recursion Under Condition I}\Lsec{Recursion in the condition I}
In this appendix we derive the extended recursion formula for the probability distribution of $Z_n$ under Condition I. We write the probability distribution of $Z_n$ under Condition I as
\be
P_n(\ell)\equiv \mathbb{P}(Z_{n}=\ell) \quad (n\geq 0),
\ee
where $P_0(\ell)=\mathbb{P}(Z_{0}=\ell)=\delta_{\ell,1}$ by definition. The key to deriving the formula is the recursive structure~\Req{f_n recursion} of the probability generating functions, which reads
\be
f_n(s)=f(f_{n-1}(s)).
\Leq{fn_recursion}
\ee
Evaluating the $\ell$th derivative of both sides of~\Req{fn_recursion}
yields 
\be
\frac{d^\ell}{ds^\ell}f_n(s)
=\frac{d^{\ell-1}}{ds^{\ell-1}}f'(f_{n-1}(s))f'_{n-1}(s).
\Leq{fn_recursion_derivative}
\ee
When the offspring distribution is Poisson with mean $\lambda$,
one has $f'(s)=\lambda f(s)$, on the basis of which the equality
$f'(f_{n-1}(s))=\lambda f(f_{n-1}(s))=\lambda f_n(s)$ holds,
where the second equality is due to~\eqref{eq:fn_recursion}~\cite{Lombard1961}.
One can then represent the right-hand side of \Req{fn_recursion_derivative} as a polynomial of the derivatives of $f_{n}f'_{n-1}$ with simple coefficients.
Since the probability distribution is derived from the probability generating function via
\be
P_n(\ell)=\frac{1}{\ell!}\left. \frac{d^\ell}{ds^\ell}f_n(s) \right|_{s=0},
\Leq{PGF2PD}
\ee
the derivatives of $f_{n}$ and $f'_{n-1}$ are converted to the probability distributions through \Req{PGF2PD}, which allows us to convert~\Req{fn_recursion_derivative} into a recursion relation between $P_n$ and $P_{n-1}$.

We extend the above-mentioned derivation in~\cite{Lombard1961}
via assuming the condition $f'(s)=a(f(s))^m$ for the probability generating
function $f$ of the offspring distribution, with some $a$ and
a positive integer $m$.
Noting that $f(1)=1$ and $f'(1)=\lambda$ hold, one should have $a=\lambda$, 
so that our assumption now reads 
\be
f'(s)=\lambda(f(s))^m.
\Leq{f'-assumption}
\ee
As mentioned above, the Poisson offspring distribution satisfies
this assumption with $m=1$, which is the case dealt with in~\cite{Lombard1961}. 
The geometric offspring distribution satisfies it with $m=2$.
More generally, for $m\in\{2,3,\ldots\}$,
the offspring distribution which is negative binomial
with parameters $r=1/(m-1)$ and $p\in(0,1)$ satisfies it
with $\lambda=r(1-p)/p$.
We show this fact. 

The probability function 
of the negative binomial distribution~\footnote{Here we adopt an extended definition of the negative binomial distribution which allows non-integer values of $r$.} with 
the parameter $r>0$
and the mean $\lambda=r(1-p)/p$ is given by
\begin{equation}
  P_{\mathrm{NB}}(x)={x+r-1\choose x}p^r(1-p)^x,
  \quad x\in\{0,1,2,\ldots\}.
\end{equation}
Its probability generating function is given by 
\be
f_{\rm NB}(s)=\left(\frac{r}{r+\lambda-\lambda s}\right)^r.
\Leq{f_NB}
\ee
Here the parameter $p \in (0,1]$ is usually called success probability and it is related to the mean $\lambda$ as
\be
p=\frac{r}{r+\lambda}.
\ee
Letting $r=1$ in the negative binomial distribution yields 
the geometric distribution 
\be
P_{\rm geo}(x)=p(1-p)^x, \quad x\in\{0,1,2,\ldots\}.
\Leq{P_geo}
\ee
It is readily seen that 
$f_{\mathrm{NB}}$ with $m=1+1/r\in\mathbb{N}$ satisfies \Req{f'-assumption}. 

The assumption~\Req{f'-assumption} allows us to rewrite the right-hand side of \Req{fn_recursion_derivative} as
\be
&&
\hspace{-0.0cm}
\frac{d^{\ell-1}}{ds^{\ell-1}}f'(f_{n-1}(s))f'_{n-1}(s)
=
\lambda
\frac{d^{\ell-1}}{ds^{\ell-1}}
\lb f_{n}(s)\rb^{m} f'_{n-1}(s)
\no \\ &&
=
\lambda 
\sum_{i_1+i_2+\ldots +i_{m+1}=\ell-1}
\binom{\ell-1}{i_1,i_2,\ldots,i_{m+1}}
\no \\ &&
\times f_n^{(i_1)}(s)\cdots f_n^{(i_m)}(s)f_{n-1}^{(i_{m+1}+1)}(s),
\ee
where the first equality is obtained by combining \Reqs{fn_recursion}{f'-assumption}, and where the last line is derived using the generalized Leibniz rule: $f^{(i)}$ denotes the $i$th derivative of $f$, and $\sum_{i_1+i_2+\ldots +i_{m+1}=\ell-1}$ denotes the summation over the set $\{ (i_1,i_2,\ldots,i_{m+1})\in \{0,1,\ldots, \ell-1\}^{m+1} \mid i_1+i_2+\cdots i_{m+1}=\ell-1  \}$ of multi-indices. Inserting this into \Req{fn_recursion_derivative}, dividing both sides by $\ell!$, and substituting $s=0$, we arrive at the following recursion:
\be
&&
P_n(\ell)
=
\frac{\lambda}{\ell} 
\sum_{i_1+i_2+\ldots i_{m+1}=\ell-1}
(i_{m+1}+1)
\no \\ &&
\times P_n(i_1)\cdots P_n(i_m)P_{n-1}(i_{m+1}+1),
\Leq{P_recursion}
\ee
which is valid for any $n\geq 2$.
For $n=1$, $P_{1}(\ell)=\mathbb{P}(Z_1=\ell)=\mathbb{P}(X_1=\ell)$ is nothing
but the offspring distribution.
The boundary condition $P_n(0)$ necessary for \Req{P_recursion} is obtained from the recursion as well:
\be
P_{n}(0)=f_{n}(0)=f(f_{n-1}(0))=f(P_{n-1}(0)).
\Leq{initcond}
\ee
By recursively evaluating \Reqs{P_recursion}{initcond}, $( P_n)_n$ is determined without ambiguity.
Since the offspring distribution that we assume (Poisson or negative binomial)
is with infinite support, 
one needs to truncate the range of $\ell$ in calculating $(P_n)_n$. 

\section{Recursion Under Condition II}\Lsec{Recursion in the condition II}
In this appendix we derive the extended recursion formula
for the probability distribution of $Z_n$ 
under Condition II. 
We write the probability distribution of $Z_n$ under Condition II as
\be
Q_n(\ell)\equiv \mathbb{P}(Z_{n}=\ell)\quad (n\geq 0).
\ee
The corresponding probability generating function is denoted as $g_n(s)$ hereafter. 
Using  this notation and the underlying GW process \NReq{GW process}, the following recurrence relation can be derived:
\begin{align}
g_n(s) 
= g_0\circ(\overbrace{f\circ f\circ\cdots\circ f}^{\text{$n$ times}})(s)
= g_0(f_n(s))\quad (n\geq 0),
\Leq{PGF_nested_II}
\end{align}
where $f(s)$ denotes the probability generating function of $X_n^{(i)}$ and where $f_n(s)$ is the probability generating function of $Z_n$ conditional on $Z_0=1$ as in Condition I.
Evaluating the $\ell$th derivative of both sides of \Req{PGF_nested_II} leads to
\be
\frac{d^\ell}{ds^\ell}g_n(s)
=\frac{d^{\ell-1}}{ds^{\ell-1}}g_0'(f_{n}(s))f'_{n}(s),
\Leq{gn_recursion_derivative}
\ee
Similarly to \Req{f'-assumption}, for the probability generating function
$g_0$ of $Z_0$ we assume 
\be
g_0'(s)=\lambda_0(g_0(s))^q,
\Leq{g0'-assumption}
\ee
with a positive integer $q$,
where $\lambda_0>0$ is the mean of $Z_0$.
It amounts to assuming that $Z_0$ is either Poisson ($q=1$) 
or negative binomial ($q\in\{2,3,\ldots\}$). 
The case $q=m=1$ was dealt with in~\cite{Dietz1965}. 
Equation~\Req{gn_recursion_derivative} is thus rewritten as
\be
&&
\frac{d^\ell}{ds^{\ell-1}}g_0'(f_{n}(s))f'_{n}(s)
\no \\ &&
=
\lambda_0 
\sum_{i_1+i_2+\ldots +i_{q+1}=\ell-1}
\binom{\ell-1}{i_1,i_2,\ldots,i_{q+1}}
\no \\ &&
\times g_n^{(i_1)}(s)\cdots g_n^{(i_q)}(s)
f_{n}^{(i_{q+1}+1)}(s),
\ee
Inserting this into \Req{gn_recursion_derivative} and dividing both sides by $\ell!$, we obtain
\be
&&
Q_n(\ell)
=
\frac{\lambda_0}{\ell} 
\sum_{i_1+i_2+\ldots i_{q+1}=\ell-1}
(i_{q+1}+1)
\no \\ &&
\times Q_n(i_1)\cdots Q_n(i_q)P_n(i_{q+1}+1),
\Leq{Q_recursion}
\ee
at $s=0$. This formula holds for $n\geq 1$ and the boundary condition $Q_n(0)$ is given by
\be
Q_{n}(0)=g_{0}(f_n(0))=g_{0}(P_n(0)).
\Leq{initcond_Q}
\ee
The solution to \Reqs{Q_recursion}{initcond_Q} is efficiently computed recursively on top of $( P_n)_n$ which are evaluated with \Reqs{P_recursion}{initcond}. These constitute the algorithmic basis for computing the distributions of $(Z_n)_n$ under Condition II. 

\section{Recursion Specific to the Geometric Distribution}\label{sec:Recursion Specific to the Geometric Distribution}
The probability generating function of the geometric distribution is given by \Req{f_NB} with $r=1$. Combining this with \Req{fn_recursion}, it is easy to show that $f_n$ has a linear fractional form and that its parameters satisfy the following recursive equation:
\be
&&
f_n(s)=\frac{c_n+d_n s}{a_n-b_n s},
\Leq{fn_geo}
\\
&&
\left(
\begin{array}{c}
 a_{n+1}     \\
b_{n+1}      \\
c_{n+1}     \\ 
 d_{n+1}   
\end{array}
\right)
=
\left(
\begin{array}{c}
 (1+\lambda) a_{n} -  \lambda c_n     \\
 (1+\lambda) b_{n} + \lambda d_n     \\
 a_{n}     \\ 
 -b_{n}    
\end{array}
\right),
\Leq{recursion_geoparam}
\ee
with the initial condition $(a_1,b_1,c_1,d_1)=(1+\lambda,\lambda,1,0)$. By solving this numerically, the probability generating function for any $n$ can be easily computed. Note that directly iterating \Req{recursion_geoparam} causes the parameter values to grow rapidly, and hence it is technically better to normalize them by multiplying an appropriate identical factor to both the numerator and denominator of \Req{fn_geo}. Given $f_n$, the distribution of $Z_n$ can be computed by a power series expansion of the denominator of  \Req{fn_geo}, yielding
\begin{align}
P_n(0)=\frac{c_n}{a_n}, 
\quad
P_n(\ell)=\frac{b_n^{\ell-1}}{a_n^{\ell+1}}(c_n b_n+a_n d_n)~~(\ell \geq 1).
\end{align}

\section{Proof of Proposition~\protect\ref{prop:tail}}
\label{sec:prf_tail}
\begin{proof}
  Assume that $M_X(t)=\mathbb{E}(e^{tX})$ is finite at some $t>0$.
  From the Chernoff bound, one has, for any $x\ge0$, 
  \begin{align}
    \bar{F}_X(x)&=\mathbb{P}(X>x)\le\mathbb{P}(X\ge x)=\mathbb{P}(e^{tX}\ge e^{tx})
    \nonumber\\
    &\le\frac{\mathbb{E}(e^{tX})}{e^{tx}}=\frac{M_X(t)}{e^{tx}}.
  \end{align}
  Letting $\tau_X(t,x)=e^{tx}\bar{F}_X(x)$, the above inequality means that $\tau_X(t,x)\le M_X(t)<\infty$ for any $x\ge0$.
  
  For $z\ge0$, let
  \begin{equation}
    m_X(t,z):=\int_0^ze^{tx}dF_X(x).
  \end{equation}
  One then has $M_X(t)=\lim_{z\to\infty}m_Z(t,z)$.
  The integration-by-parts formula gives 
  \begin{align}
    m_X(t,z)
    &=\left[e^{tx}(F_X(x)-1)\right]_0^z
    -\int_0^z(F_X(x)-1)te^{tx}\,dx
    \nonumber\\
    &=t\int_0^z\tau_X(t,x)\,dx
    +\bar{F}_X(0)-\tau_X(t,z)
    \label{eq:prop2tmp}\\
    &\ge t\int_0^z\tau_X(t,x)\,dx
    +\bar{F}_X(0)-M_X(t).
  \end{align}
  Taking the limit $z\to\infty$, one obtains
  \begin{equation}
    \int_0^\infty\tau_X(t,x)\,dx
    \le\frac{2M_X(t)-\bar{F}_X(0)}{t}<\infty,
  \end{equation}
  which, together with $\tau_X(t,x)\ge0$, implies
  \begin{equation}
    \lim_{x\to\infty}\tau_X(t,x)=0,
  \end{equation}
  proving item (a).

  We next prove (b).
  Let $h_X(t):=\sup_x\tau_X(t,x)$, which is finite by assumption.
  For any $z\in[0,\infty)$ and any $t'\in[0,t)$, one has 
  \begin{align}
    \int_0^z\tau_X(t',x)\,dx
    &=\int_0^ze^{-(t-t')x}\tau_X(t,x)\,dx
    \nonumber\\
    &\le\int_0^ze^{-(t-t')x}h_X(t)\,dx
    \nonumber\\
    &=h_X(t)\frac{1-e^{-(t-t')z}}{t-t'}.
  \end{align}
  Combining this and~\eqref{eq:prop2tmp}, one obtains
  \begin{align}
    m_X(t',z)&\le t\,h_X(t)\frac{1-e^{-(t-t')z}}{t-t'}
    +\bar{F}_X(0)-\tau_X(t',z)
    \nonumber\\
    &\le t\,h_X(t)\frac{1-e^{-(t-t')z}}{t-t'}
    +\bar{F}_X(0).
  \end{align}
  Taking the limit $z\to\infty$,
  one has
  \begin{equation}
    M_X(t')\le\frac{t\,h_X(t)}{t-t'}+\bar{F}_X(0)<\infty,
  \end{equation}
  proving item (b).
\end{proof}

\section{Order of $c_k$ in the asymptotic $\lambda \downarrow 1$}
\label{sec:order_c_k}

The objective here is to evaluate the order of the coefficients $\{c_k\}$ in \eqref{eq:power_series_h_n}.

\begin{lem}
\label{lem:1}
Under Assumption~\ref{asm:offspring}, let 
\begin{align}
\bar{W}_m = \frac{(\lambda-1)}{\lambda^m}Z_m,
\label{eq:bar_W_m}
\end{align}
where $Z_m$ is a random variable defined by \eqref{eq:GW process} with $Z_0 \equiv 1$.
Assume that the cumulant generating function $K_{\bar{W}_{m}}(t)$ of $\bar{W}_m$ exists, i.e., 
\begin{align}
K_{\bar{W}_{m}}(t) = \sum_{k=1}^{\infty} \frac{\bar{\kappa}_k^{(m)}}{k!}t^k
\label{eq:power series bar_K_m}
\end{align}
is absolutely convergent on an open interval containing $t=0$.
Then, for any $m$, the first and second cumulants of $\bar{W}_m$ are given by 
\begin{align}
\bar{\kappa}_1^{(m)} &= \lambda-1, \label{eq:kappa_1_m} \\
\bar{\kappa}_2^{(m)} &= \frac{\kappa_2 (\lambda-1)}{\lambda}\left( 1 - \frac{1}{\lambda^m} \right). \label{eq:kappa_2_m}
\end{align}
Moreover, for $k \geq 3$ the order of the $k$th cumulant $\bar{\kappa}_k^{(m)}$ is $o\left( (\lambda-1)^2 \right)$ in the asymptotic regime $\lambda \downarrow 1$.
\end{lem}

\begin{proof}
Substituting \eqref{eq:CGF_off} and \eqref{eq:power series bar_K_m} into the functional equation~\eqref{eq:recK}, expanding both sides into power series in $t$, and comparing the coefficients of $t$-proportional terms, one has
\begin{align}
\lambda \bar{\kappa}_1^{(m+1)} = \lambda \bar{\kappa}_1^{(m)}.
\end{align}
From this recurrence relation and $\bar{\kappa}_1^{(1)} = \lambda-1$, one obtains \eqref{eq:kappa_1_m}.

Similarly, comparing the coefficients of $t^2$-proportional terms, one has
\begin{align}
\frac{\lambda^2}{2} \bar{\kappa}_2^{(m+1)} = \frac{\lambda}{2} \bar{\kappa}_2^{(m)} + \frac{\kappa_2}{2} \left\{ \bar{\kappa}_1^{(m)} \right\}^2, 
\end{align}
which is simplified to
\begin{align}
\bar{\kappa}_2^{(m+1)} = \frac{1}{\lambda} \bar{\kappa}_2^{(m)} + \frac{\kappa_2 (\lambda-1)^2}{\lambda^2}.
\end{align}
By solving this recurrence relation with $\bar{\kappa}_2^{(1)} = \left( \frac{\lambda-1}{\lambda} \right)^2 \kappa_2$, one obtains \eqref{eq:kappa_2_m}.

We prove the remaining claim via induction on $k$.

For $k=3$, comparing the coefficients of $t^3$-proportional terms, one has
\begin{align}
\frac{\lambda^3}{6} \bar{\kappa}_3^{(m+1)} = \frac{\lambda}{6} \bar{\kappa}_3^{(m)} + \frac{\kappa_2}{2} \bar{\kappa}_1^{(m)} \bar{\kappa}_2^{(m)} + \frac{\kappa_3}{6} \left\{ \bar{\kappa}_1^{(m)} \right\}^3.
\label{eq:comparing_coef_of_t3}
\end{align}
From \eqref{eq:kappa_1_m} and \eqref{eq:kappa_2_m}, \eqref{eq:comparing_coef_of_t3} is simplified to
\begin{align}
\bar{\kappa}_3^{(m+1)} = \frac{1}{\lambda^2} \bar{\kappa}_3^{(m)} 
+ \frac{3\kappa_2^2 \sum_{i=1}^{m} \lambda^{i-1}}{\lambda^{m+4}} (\lambda-1)^3
+ \frac{\kappa_3 (\lambda-1)^3}{\lambda^3}.
\end{align}
By solving this recurrence relation with $\bar{\kappa}_3^{(1)} = \left( \frac{\lambda-1}{\lambda} \right)^3 \kappa_3$, one obtains 
\begin{align}
\bar{\kappa}_3^{(m)} =& \frac{1}{\lambda^{2(m-1)}} \left( \frac{\lambda-1}{\lambda} \right)^3 \kappa_3 \nonumber \\
&+ \frac{(\lambda-1)\sum_{i=1}^{2(m-1)}\lambda^{i-1}}{\lambda^{2(m-1)}} \beta_3,
\label{eq:bar_kappa_3_m}
\end{align}
where
\begin{align}
\beta_3 =& \frac{\lambda^{2}}{(\lambda-1)(\lambda+1)} \nonumber \\
& \times \left\{ \frac{3\kappa_2^2 \sum_{i=1}^{m} \lambda^{i-1}}{\lambda^{m+4}} (\lambda-1)^3
+ \frac{\kappa_3 (\lambda-1)^3}{\lambda^3} \right\}.
\label{eq:beta_3}
\end{align}
From \eqref{eq:beta_3} and Assumption~\ref{asm:offspring}, one has $\beta_3=o(\lambda-1)$ in the asymptotic regime $\lambda \downarrow 1$.
Thus, from \eqref{eq:bar_kappa_3_m} and Assumption~\ref{asm:offspring}-(\ref{asm:iii}), the third cumulant $\bar{\kappa}_3^{(m)}$ is $o\left( (\lambda-1)^2 \right)$, showing that the claim holds for $k=3$.

For $k \geq 4$, assume that, for all $l \in \{ 3, \cdots, k-1 \}$, the $l$th cumulant $\bar{\kappa}_l^{(m)}$ is $o\left( (\lambda-1)^2 \right)$ in the asymptotic regime $\lambda \downarrow 1$.
Comparing the coefficients of $t^k$-proportional terms, one has
\begin{align}
\frac{\lambda^k}{k!} \bar{\kappa}_k^{(m+1)} = \frac{\lambda}{k!} \bar{\kappa}_k^{(m)} + \frac{\kappa_2}{2} \sum_{l=1}^{k-1} \frac{\bar{\kappa}_l^{(m)}}{l!} \frac{\bar{\kappa}_{k-l}^{(m)}}{(k-l)!} + (**),
\label{eq:recurrence_relation_for_kappa_k}
\end{align}
where the part $(**)$ is a linear combination of terms of the form $\kappa_\nu \prod_{j=1}^{\nu}\bar{\kappa}_{i_j}^{(m)}$, where $3\leq \nu \leq k$, $\sum_{j=1}^{\nu}i_j=k$, and $1 \leq i_j \leq k-\nu+1$.
Simplifying \eqref{eq:recurrence_relation_for_kappa_k} yields
\begin{align}
\bar{\kappa}_k^{(m+1)} = \frac{1}{\lambda^{k-1}} \bar{\kappa}_k^{(m)} + \left\{ \frac{\kappa_2}{2\lambda^k} \sum_{l=1}^{k-1} \binom{k}{l} \bar{\kappa}_l^{(m)} \bar{\kappa}_{k-l}^{(m)} + \frac{k!}{\lambda^{k}} (**) \right\}.
\label{eq:recurrence_relation_for_kappa_k_2}
\end{align}
By solving this recurrence relation with $\bar{\kappa}_k^{(1)} = \left(\frac{\lambda-1}{\lambda}\right)^k \kappa_k$, one obtains
\begin{align}
\bar{\kappa}_k^{(m)} =& \frac{1}{\lambda^{(k-1)(m-1)}} \left(\frac{\lambda-1}{\lambda}\right)^k \kappa_k \nonumber \\
&+ \frac{(\lambda-1)\sum_{i=1}^{(k-1)(m-1)}\lambda^{i-1}}{\lambda^{(k-1)(m-1)}} \beta_k,
\end{align}
where
\begin{align}
\beta_k =& \frac{\lambda^{k-1}}{(\lambda-1)\sum_{i=1}^{k-1}\lambda^{i-1}} \nonumber \\
& \times \left\{ \frac{\kappa_2}{2\lambda^k} \sum_{l=1}^{k-1} \binom{k}{l} \bar{\kappa}_l^{(m)} \bar{\kappa}_{k-l}^{(m)} + \frac{k!}{\lambda^{k}} (**) \right\}.
\label{eq:beta_k}
\end{align}
Regarding the first term inside the braces on the right-hand side of \eqref{eq:beta_k}, from \eqref{eq:kappa_1_m}, \eqref{eq:kappa_2_m}, the inductive hypothesis, and Assumption~\ref{asm:offspring}, its order is given as 
\begin{align}
\frac{\kappa_2}{2\lambda^k} \sum_{l=1}^{k-1} \binom{k}{l} \bar{\kappa}_l^{(m)} \bar{\kappa}_{k-l}^{(m)} = o\left( (\lambda-1)^3 \right),
\label{eq:leading_order_first_term_2}
\end{align}
which arises from the terms of the form $\kappa_2 \bar{\kappa}_1^{(m)} \bar{\kappa}_{k-1}^{(m)}$.
Similarly, the second term inside the braces of \eqref{eq:beta_k} is given as 
\begin{align}
\frac{k!}{\lambda^{k}}(**) = o\left( (\lambda-1)^2 \right),
\label{eq:leading_order_remaining_2}
\end{align}
which arises from the term of the form $\kappa_k \{ \bar{\kappa}_1^{(m)} \}^k$.
From \eqref{eq:beta_k}, \eqref{eq:leading_order_first_term_2}, and \eqref{eq:leading_order_remaining_2}, 
\begin{align}
\beta_k = o(\lambda-1).
\label{eq:leading_order_beta_k}
\end{align}
It follows from \eqref{eq:leading_order_beta_k} and Assumption~\ref{asm:offspring}-(\ref{asm:iii}) that the $k$th cumulant $\bar{\kappa}_k^{(m)}$ is $o\left( (\lambda-1)^2 \right)$ in the asymptotic regime $\lambda \downarrow 1$.

Via induction, we have proved the lemma.
\end{proof}

Since $\bar{\kappa}_1^{(m)} = \lambda-1 > 0$,
\begin{align}
\frac{t}{K_{\bar{W}_{m}}(t)} = \sum_{k=0}^\infty \tilde{c}_k^{(m)} t^{k}
\label{eq:t_over_K_m}
\end{align} 
exists, and the coefficients $\{ \tilde{c}_k^{(m)} \}$ are given by the following recurrence relation~\cite{Wilf1990}:
\begin{align}
\tilde{c}_0^{(m)} &= \frac{1}{\bar{\kappa}_1^{(m)}}, \\
\tilde{c}_k^{(m)} &= - \frac{1}{\bar{\kappa}_1^{(m)}} \left( \sum_{i=1}^{k} \frac{\bar{\kappa}_{i+1}^{(m)}}{(i+1)!} \tilde{c}_{k-i}^{(m)} \right).
\label{eq:recurrence_relation_tilde_c_k}
\end{align} 
Using Lemma~\ref{lem:1}, $\tilde{c}_0^{(m)}$ and $\tilde{c}_1^{(m)}$ are obtained as
\begin{align}
\tilde{c}_0^{(m)} &= \frac{1}{\lambda-1}, \label{eq:tilde_c_0} \\
\tilde{c}_1^{(m)} &= - \frac{\kappa_2 \sum_{i=1}^{m} \lambda^{i-1}}{2 \lambda^{m+1}},  \label{eq:tilde_c_1}
\end{align} 
respectively.

Next, we prove the following lemma.
\begin{lem}
\label{lem:2}
Assume Assumption~\ref{asm:offspring}.
Then, for $k \geq 2$, the coefficient $\tilde{c}_k^{(m)}$ given by \eqref{eq:recurrence_relation_tilde_c_k} with $\tilde{c}_0^{(m)} = \frac{1}{\lambda-1}$ is $o(1)$ in the asymptotic regime $\lambda \downarrow 1$.
\end{lem}

\begin{proof}
We prove the lemma via induction on $k$.

For $k=2$, from \eqref{eq:recurrence_relation_tilde_c_k} and \eqref{eq:kappa_1_m}, 
\begin{align}
\tilde{c}_2^{(m)} &= - \frac{1}{\lambda-1} \left( \frac{\bar{\kappa}_{2}^{(m)}}{2!} \tilde{c}_{1}^{(m)} + \frac{\bar{\kappa}_{3}^{(m)}}{3!} \tilde{c}_{0}^{(m)} \right).
\end{align}
Thus, it follows from \eqref{eq:tilde_c_0}, \eqref{eq:tilde_c_1}, and Lemma~\ref{lem:1} that the coefficient $\tilde{c}_2^{(m)}$ is $o(1)$ in the asymptotic regime $\lambda \downarrow 1$, showing that the claim holds for $k=2$.

For $k \geq 3$, assume that, for all $l \in \{ 2, \cdots, k-1 \}$, the coefficient $\tilde{c}_l^{(m)}$ is $o(1)$ in the asymptotic regime $\lambda \downarrow 1$.
The coefficient $\tilde{c}_k^{(m)}$ is given by \eqref{eq:recurrence_relation_tilde_c_k}.
Regarding the sum inside the parentheses on the right-hand side of \eqref{eq:recurrence_relation_tilde_c_k}, it follows from \eqref{eq:recurrence_relation_tilde_c_k}, \eqref{eq:tilde_c_0}, \eqref{eq:tilde_c_1}, the inductive hypothesis, and Lemma~\ref{lem:1} that the sum is $o(\lambda-1)$, which arises from $\frac{\bar{\kappa}_{k+1}^{(m)}}{(k+1)!} \tilde{c}_{0}^{(m)}$.
Therefore, by \eqref{eq:recurrence_relation_tilde_c_k}, the coefficient $\tilde{c}_k^{(m)}$ is $o(1)$.

Via induction, we have proved the claim.
\end{proof}

Using \eqref{eq:t_over_K_m}, $h_m(t)$ defined by \eqref{eq:def_h_n} can be written as
\begin{align}
h_m(t) = \left( \sum_{k=0}^\infty \tilde{c}_k^{(m)} t^{k-1} \right) - \frac{1}{\lambda-1} \frac{1}{t} + \frac{\kappa_2^*}{2\lambda(\lambda-1)}.
\label{eq:def_h_n_2}
\end{align}
Comparing \eqref{eq:def_h_n_2} with \eqref{eq:power_series_h_n}, one obtains the following correspondence between the coefficients:
\begin{align}
c_0 &= \tilde{c}_0^{(m)} - \frac{1}{\lambda-1} = 0, \tag{\ref{eq:c_0}} \\
c_1 &= \tilde{c}_1^{(m)} + \frac{\kappa_2^*}{2\lambda(\lambda-1)} \nonumber \\
&= \frac{\kappa_2^*}{2\lambda(\lambda-1)} - \frac{\kappa_2 \sum_{i=1}^{m} \lambda^{i-1}}{2 \lambda^{m+1}} , \tag{\ref{eq:c_1}} \\
c_k &= \tilde{c}_k^{(m)}, \quad k\geq 2. \label{eq:c_k}
\end{align}
Therefore, it follows from Lemma~\ref{lem:2} that the coefficient $c_k$ for $k \geq 2$ is $o(1)$ in the asymptotic regime $\lambda \downarrow 1$.

\bibliography{Reference_AMoPGWP}

\end{document}